\documentclass[%
 ,reprint%
 ,onecolumn%
 ,notitlepage
 ,aps%
 ,a4paper,superscriptaddress,nofootinbib%
]
{revtex4-2}

\usepackage[utf8]{inputenc}

\usepackage{enumitem}

\usepackage{amsmath,amssymb,bm,cancel}
\usepackage{amsthm}
\usepackage{mathtools}
\usepackage[dvipsnames]{xcolor}
\usepackage{float}

\usepackage{hyperref}
\usepackage[capitalise]{cleveref}
\crefformat{theorem}{Thm.~#2#1#3}
\crefformat{lemma}{Lem.~#2#1#3}
\crefformat{definition}{Defn.~#2#1#3}
\crefformat{equation}{(#2#1#3}
\crefformat{appendix}{App.~#2#1#3}
\crefformat{question}{Question~#2#1#3}
\crefformat{proposition}{Prop~#2#1#3}

\theoremstyle{plain}
\newtheorem{theorem}{Theorem}[section]
\newtheorem{proposition}[theorem]{Proposition}
\newtheorem{corollary}[theorem]{Corollary}
\newtheorem{lemma}[theorem]{Lemma}

\theoremstyle{definition}
\newtheorem{definition}[theorem]{Definition}

\theoremstyle{remark}
\newtheorem{example}[theorem]{Example}
\newtheorem{remark}[theorem]{Remark}

\DeclareMathOperator{\interior}{int}
\DeclareMathOperator{\Div}{div}

\renewcommand{\L}{\mathcal{L}}
\newcommand{\U}{\mathcal{U}}

\newcommand{\R}{\mathbb{R}}
\newcommand{\T}{\mathbb{T}}

\newcommand{\tB}{{\tilde{B}}}
\newcommand{\tmu}{{\tilde{\mu}}}

\begin{document}

\title{Existence of global symmetries of divergence-free fields with first integrals}

\date{March 28, 2023}

\author{David Perrella}
\email[]{david.perrella@uwa.edu.au}
\affiliation{The University of Western Australia, 35 Stirling Highway, Crawley WA 6009, Australia}

\author{Nathan Duignan}
\affiliation{School of Mathematics and Statistics, University of Sydney, NSW 2050,
Australia}

\author{David Pfefferlé}
\affiliation{The University of Western Australia, 35 Stirling Highway, Crawley WA 6009, Australia}

\begin{abstract}
The relationship between symmetry fields and first integrals of divergence-free vector fields is explored in three dimensions in light of its relevance to plasma physics and magnetic confinement fusion. A Noether-type Theorem is known: for each such symmetry, there corresponds a first integral. The extent to which the converse is true is investigated. In doing so, a reformulation of this Noether-type Theorem is found for which the converse holds on what is called the toroidal region. Some consequences of the methods presented are quick proofs of the existence of flux coordinates for magnetic fields in high generality; without needing to assume a symmetry such as in the cases of magneto-hydrostatics (MHS) or quasi-symmetry.
\end{abstract}

\maketitle



\section{Introduction}

There is a growing interest in divergence-free vector fields $B$ that admit either a symmetry or a first integral. A symmetry $u$ is a divergence-free vector field and a first integral $\psi$ is a function which respectively satisfy
\[[u,B] = 0, \qquad  B\cdot \nabla \psi = 0.\]
Here $[u,B]$ is the Lie-bracket of vector fields $u$ and $B$. In the context of plasma physics, one considers magneto-hydrostatics (MHS) under which $B$ satisfies
\[J \times B = \nabla p, \qquad J = \nabla \times B\]
where $p$ is the pressure and $J$ is the current density. In this case, $J$ is a symmetry of $B$ and $p$ is a first integral of both $B$ and $J$. Another class of magnetic fields, called \emph{quasi-symmetric}, arise when there exists a divergence-free vector field $u$ which possesses $B$ as a symmetry and $\|B\|^2$ as a first integral~\cite{burby2020some,rodriguez-helander-2020}. In magnetic confinement fusion devices, transport of heat and particles along magnetic field-lines is generally several orders of magnitude faster than across. Promoting first integrals is a way of guaranteeing temperature and density gradients because then the field-lines lie on lower-dimensional manifolds. Without substantial gradients, the core of a fusion device would not be hot enough to sustain fusion reactions and the outer region cold enough for the device's walls to resist damage. 

For any $B$ admitting both a symmetry and an integral, the field is integrable in the broad sense \cite{burby2021integrability} as defined in \cite{bogoyavlenskijExtendedIntegrabilityBiHamiltonian1998}. On the other hand, divergence-free fields generally admit no symmetries nor integrals and are chaotic \cite{Arnold65b,cardonaConstructingTuringComplete2021}.

Between the cases of chaos and integrability lies magnetic fields $B$ which admit either an integral or a symmetry, but a priori not both. For example, if one has a vacuum magnetic field, which is a divergence-free vector field $B$ satisfying $\nabla \times B = 0$, and it is assumed to have nested flux surfaces, it is not immediately clear whether or not $B$ also possesses a symmetry.

It has been established \cite{hallerReductionThreedimensionalVolumepreserving1998,burby2021integrability} that a divergence-free vector field $B$ in 3D admits a Noether theorem: loosely, this states that for each symmetry field of $B$, there corresponds a first integral of $B$. However, the converse is false: there exists $B$ which has a first integral but no corresponding symmetry (see for instance \cref{counterexample}). Fortunately, a partial converse is available. If $B$ is non-vanishing and has a first integral $\psi$, as indicated in \cite{borisovConservationLawsHierarchy2008}, an extension of an argument due to Kolmogorov \cite{sinaiIntroductionErgodicTheory1976} shows that in a neighbourhood $U$ of a compact, regular level-surface of $\psi$, there exists coordinates $(\theta_1,\theta_2,\psi) : U \to \mathbb{T}^2\times \R$ and a smooth positive function $f > 0$ such that 
\[ B/f = a(\psi) \partial_{\theta_1} + b(\psi) \partial_{\theta_2}.\]
Thus, the vector field $u = b(\psi)\partial_{\theta_1} - a(\psi)\partial_{\theta_2}$ on $U$, for example, is a symmetry for the rescaled field $B/f$. In this case, $B$ is said to be semi-rectified in this neighborhood $U$ \cite{perrella2022rectifiability}. However, this converse is local to a given regular level surface of $\psi$.

The main result of this paper is a \emph{converse conformal Noether theorem for divergence-free vector fields}; while $B$ may not admit a symmetry, a rescaling of $B$ does. We call such a symmetry a \emph{conformal symmetry} of $B$. The proof is independent of Kolmogorov's original ergodic methods and is more global than the above construction. The result takes place on a 3-manifold $M$ with (possibly empty) boundary and is centred around flux systems.
\begin{definition}\label{def:fluxsystem}
A \emph{flux system} is a triple $(B,\nu,\mu)$ where $B$ is a vector field on $M$, $\nu$ is a closed 1-form on $M$, and $\mu$ is a volume form on $M$, such that
\[ \nu(B) = 0, \qquad \L_B \mu = 0,\]
where $\L_B$ denotes the Lie derivative with respect to $B$.
We call $(B,\nu,\mu)$ a \emph{tangential flux system} if $B$ is tangent to $\partial M$ and the kernel of $\nu$ contains all vectors tangent to $\partial M$.
\end{definition}
For most applications $\nu = d\psi$ will be exact so that $\psi$ will be a first integral of $B$. This occurs automatically if the first cohomology of $M$ is trivial, $H^1(M) = 0$. Our main result is split into two parts. The first part highlights the role played by \emph{adapted 1-forms}.
\begin{definition}\label{def:adapted_one-form}
A 1-form $\eta \in \Omega^1(M)$ is said to be \emph{adapted to a flux system $(B,\nu,\mu)$} if
\[\eta(B) > 0, \qquad d\eta\wedge \nu = 0.\]
\end{definition}
Clearly only non-vanishing $B$ could admit such a 1-form. The first part of the main result is stated as follows.
\begin{theorem}\label{thm:uexistsGivenEta}
Let $(B,\nu,\mu)$ be a flux system on $M$. Assume that there exists a 1-form $\eta$ adapted to the flux system. Then, the vector field $u$ satisfying
\[\iota_u \mu = \frac{\nu \wedge \eta}{\eta(B)}\]
is the unique vector field $u$ such that
\[
    \iota_{u} \iota_{B} \mu = \nu, \qquad \eta(u) = 0
\]
and the vector fields $u$ and rescaling $\tB = B/\eta(B)$
\begin{enumerate}
    \item are linearly independent at $x \in M$ whenever $\nu|_x \neq 0$,
    \item satisfy $\nu(\tB) = 0 = \nu(u)$,
    \item commute $[u,\tB] = 0$,
    \item and preserve the volume form $\tmu := \eta(B) \mu$.
\end{enumerate}
\end{theorem}
\cref{thm:uexistsGivenEta} is constructive and global. It shows that, given a flux system $(B,\nu,\mu)$, one can use the 1-form $\eta$, if it exists, to uniquely generate a vector field $u$ which is a symmetry for the rescaled field $\tB$. It will also be shown how adapted 1-forms relate to the theory in \cite{burby2021integrability}, which in turn can be used to prove \cref{thm:uexistsGivenEta} from a more abstract perspective.

The second part addresses an outstanding question left from the first: for a given flux system $(B,\nu,\mu)$ with $B$ having no zeroes and $\nu$ having nowhere dense critical points, does an adapted 1-form always exist? We provide a partial answer to this question: the semi-global existence of an adapted 1-form. By semi-global, we mean in the sense of \cite{llibreNoteFirstIntegrals2012}, that it holds on an open subset of $M$ determined by the flux system. We call this subset the toroidal region (defined in \cref{sec:eta}). In essence, the toroidal region is the largest subset of $M$ consisting of invariant tori of $B$ and for which $B$ is non-vanishing, together with periodic orbits of $B$ neighboured by the tori. In the context of magnetic fields, these periodic orbits are called \emph{magnetic axes}. The following example illustrates the toroidal region.

\begin{example}\label{NonAutHamExample}
Consider the manifold $M = \R^2 \times \R/\mathbb{Z}$ with $x,y,t$ the respective coordinate projections of each factor. Define the flux system $(B,\nu,\mu)$ where
\[B = y\partial_x +(2x-4x^3)\partial_y + \partial_t, \qquad \nu = d\psi, \qquad \mu = dx \wedge dy \wedge dt,\]
and $\psi = \tfrac12 y^2 - x^2(1-x^2)$. In this case, the toroidal region $\mathcal{T} = M \backslash \{\psi = 0\}$ has three connected components divided by $\{\psi = 0\}$ as depicted in \cref{fig:toroidalregion}.
\begin{figure}[h!]
    \centering
    \includegraphics[scale = 0.5]{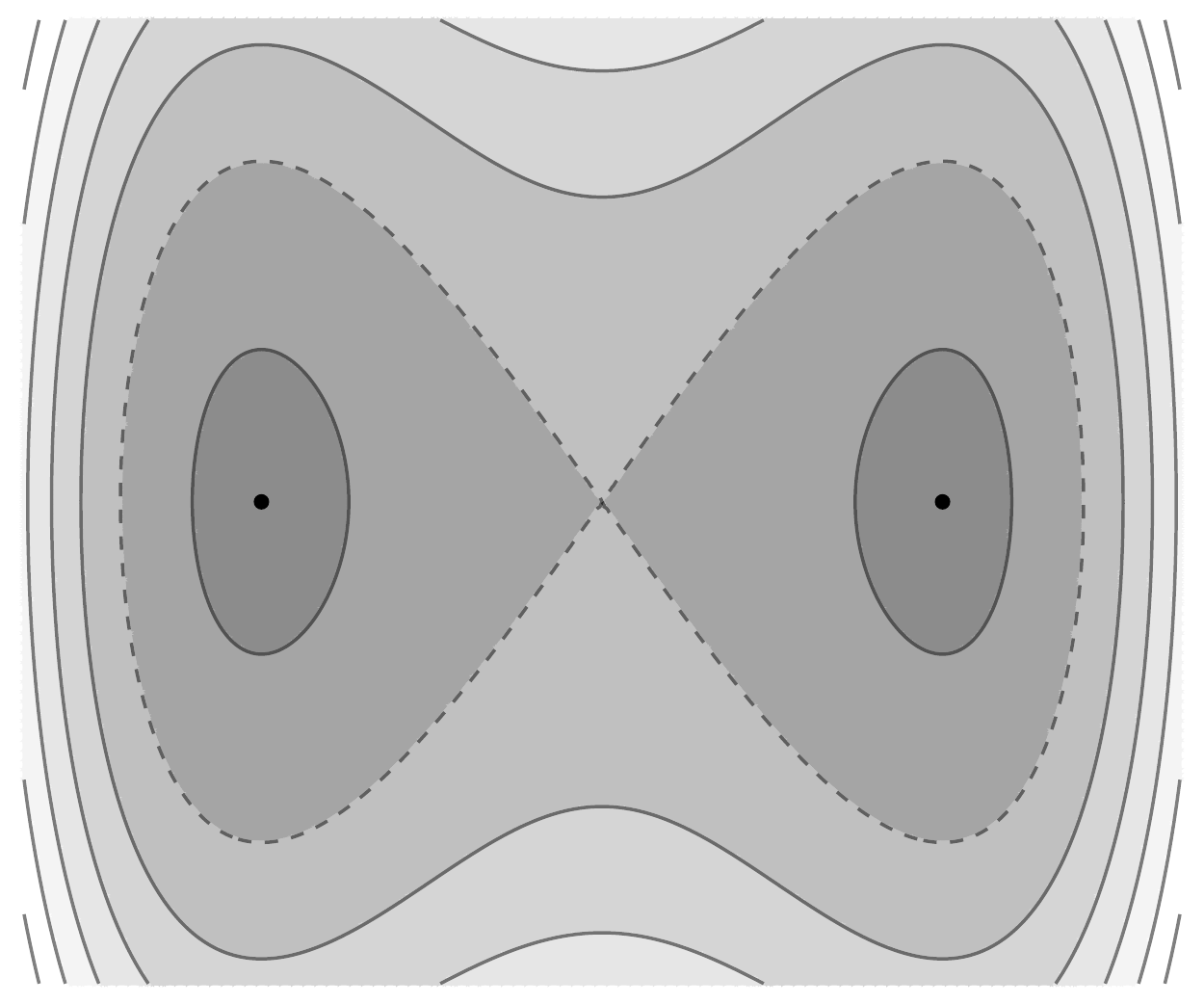}
    \caption{A constant-$t$ Poincar{\'e} section of $B$. The toroidal region is everything except for $\{\psi = 0\}$ which intersects the Poincar{\'e} section in a figure eight (dashed).}
    \label{fig:toroidalregion}
\end{figure}
\end{example}

The second part is stated as follows.

\begin{theorem}\label{thm:etaexists}
Let $(B,\nu,\mu)$ be a tangential flux system on $M$. Then, the toroidal region $\mathcal{T}$ is an open subset of $M$ for which the flux system restricted to $\mathcal{T}$ satisfies the assumptions, and therefore the conclusion, of \cref{thm:uexistsGivenEta}.
\end{theorem}

In fact, motivated by the work of \cite{Peralta-Salas2021,Kuperberg1996} we exhibit an example (\cref{theexample}) flux system $(B,d\psi,\mu)$ on a solid torus for which an adapted 1-form exists on all level sets of $\psi$ except the unique level set contained in the complement of the toroidal region. Interestingly, this flux system still possesses a symmetry and $S^*$ is in fact a torus.

The techniques used to prove \cref{thm:etaexists} also bring about the following classification of the types of topologies which can occur for the toroidal region. Some possible toroidal domains are given in \cref{fig:toroidalDomains}.

\begin{proposition}\label{classification}
Let $C$ be a connected component of $\mathcal{T}$ considered as an open submanifold of $M$ with boundary and consider the manifold interior $\interior C = C \backslash \partial C$.
\begin{enumerate}
    \item If $\interior C$ is compact, then $C = \interior C$ does not have boundary nor contain axes.
    \item Otherwise, $C$ is diffeomorphic to either $\mathbb{T}^2 \times I$ where $I$ is some interval, a solid torus $D \times \mathbb{S}^1$ (where $D$ denotes the closed unit disk), an open solid torus $\interior D \times \mathbb{S}^1$, or $\mathbb{S}^2 \times \mathbb{S}^1$. In all these cases, $\nu$ is exact on $C$.
\end{enumerate}
\end{proposition}

\begin{figure}[h!]
    \centering
    \includegraphics[width=\linewidth]{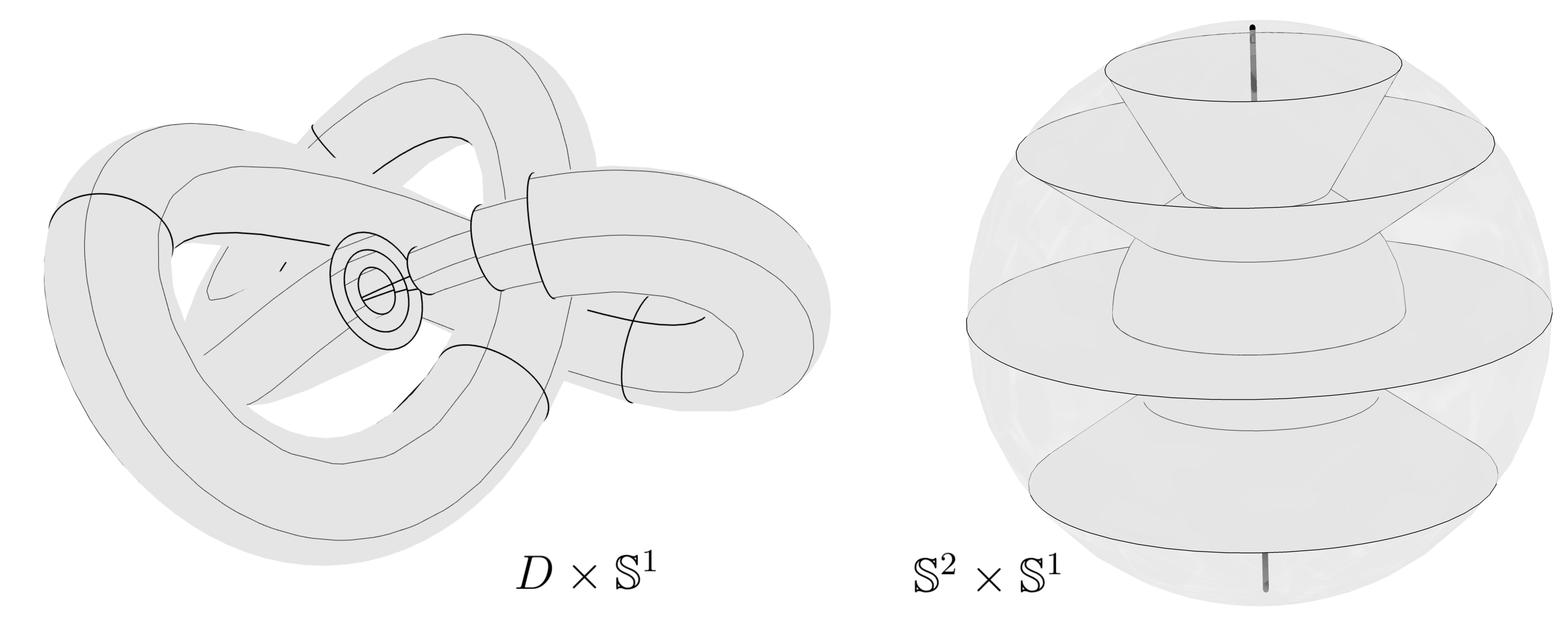}
    \caption{Two examples of possible toroidal domains. Left is a knotted toroidal domain given by the embedding of a solid torus $D\times\mathbb{S}^1$. The case $\interior D \times \mathbb{S}^1 $ and $\T^2\times I$ are analagous. On the right is the more exotic $\mathbb{S}^2\times\mathbb{S}^1$, demonstrating a toroidal domain with two axes. The annuli depicted in the right diagram are the invariant tori after identification of the inner and outer sphere producing the $\mathbb{S}^1$ factor. The two axes pass through different poles of the spheres.}
    \label{fig:toroidalDomains}
\end{figure}

The paper proceeds as follows. In \cref{sec:ifeta}, \cref{thm:uexistsGivenEta} is proved and its relation to dynamical systems theory is discussed with an example. Moreover, a counterexample to the naive converse of a Noether type theorem for divergence free fields is also given. In \cref{sec:eta}, \cref{thm:etaexists} is proved, and the toroidal region is classified. In \cref{sec:obstructiontoglobaleta} an example exhibiting an obstruction to the global existence of an adapted 1-form is discussed. In \cref{sec:flux-coordinates}, we discuss our methods in relation to plasma physics, specifically to the existence of flux-coordinates. In \cref{sec:discuss}, the results are discussed both mathematically and in relation to plasma physics. The discussion includes comments on the Noether type theorems and future work in the context of integrable systems.

\section{The Conformal Noether theorem given an adapted 1-form}\label{sec:ifeta}

The main purpose of this section is to prove \cref{thm:uexistsGivenEta}. The remaining material puts this theorem into its context as a (partial) converse to a conformal Noether Theorem. In \cref{fwdNoether}, the (forward) conformal Noether Theorem is presented and an explicit counterexample to the (non-conformal) converse Noether Theorem is provided. This is then related to \cref{thm:uexistsGivenEta} as a conformal Noether Theorem. In \cref{sec:ElementaryNoether}, an elementary proof of \cref{thm:uexistsGivenEta} is given. An alternative proof is then given in \cref{sec:presymplecticPerspective} by contextualising the theorem in the theory of presymplectic geometry. From this more abstract perspective, the theorem can be proved using a coisotropic embedding of the manifold $M$ into a symplectic manifold $\tilde{M}$. By revealing the connection to symplectic geometry, this version of the proof of \cref{thm:uexistsGivenEta} gives insight into why one should expect the conformal Noether theorem to hold. Lastly, as \cref{thm:uexistsGivenEta} relies on the existence of an adapted 1-form, \cref{sec:whenEtaIsGuaranteed} provides contexts in which adapted 1-forms are guaranteed to exist globally.

Throughout this section, we will make extensive use of the formulae \cite{Lee2012}
\begin{align*}
    \L_X \alpha &= d\iota_X\alpha + \iota_X d\alpha,\\
    \iota_{[X,Y]} \alpha &= \L_X \iota_Y\alpha - \iota_Y \L_X \alpha,
\end{align*}
for any vector fields $X,Y$ and any $k$-form $\alpha$ on $M$. In particular, if $B$ is a divergence-free vector field with respect to a volume form $\mu$, $f$ is a positive function and $u$ is a vector field, writing $\tB = B/f$ and $\tmu = f\mu$, one has $ \L_{\tB} \tmu = \L_B \mu = 0$ and therefore
\begin{equation}\label{eq:ConformalNoether}
\iota_{[u,\tB]}\tmu = \L_u \iota_\tB \tmu - \iota_\tB \L_u \tmu =  d \iota_u \iota_B \mu - \iota_\tB \L_u \tmu =  d \iota_u \iota_B \mu - (\Div_{\tmu}u) \iota_B\mu
\end{equation}
where $\Div_{\tmu}u$ is the divergence of $u$ with respect to $\tmu$: $(\Div_{\tmu}u)\tmu = \mathcal{L}_{u}\tmu$.

\subsection{The Conformal Noether theorem}\label{fwdNoether}

The forward direction of the conformal Noether theorem for divergence-free vector fields is the same observation as that of \cite[Thm.~3.1]{hallerReductionThreedimensionalVolumepreserving1998} but applied to Equation \eqref{eq:ConformalNoether}, which allows functions $f \neq 1$.

\begin{theorem}\label{thm:ForwardNoetherTheorem}
    Let $B$ be a vector field on a 3-manifold $M$ with boundary and assume it is divergence-free with respect to $\mu$. Suppose there exists a vector field $u$ on $M$ and a positive function $f\in C^{\infty}(M)$ such that $[u, B/f] = 0$ and $u$ is divergence-free with respect to $\tmu = f \mu$. Then the 1-form $\nu := \iota_u\iota_B \mu$ makes $(B,\nu,\mu)$ a flux system. 
\end{theorem}
\begin{proof}
    Equation \eqref{eq:ConformalNoether} reads
    \[\iota_{[u,\tB]}\tmu =  d \iota_u \iota_B \mu - (\Div_{\tmu}u) \iota_B\mu\]
    where $\tB = B/f$. In particular, from $[u,\tB] = 0$ and $\Div_{\tmu}u = 0$, we get that $d \iota_u \iota_B \mu = 0$. Hence, $\nu = \iota_u \iota_B \mu$ is closed. Moreover, $\nu(B) = 0$ (and also $\nu(u) = 0$) making $(B,\nu,\mu)$ a flux system.
\end{proof}

The $u$ in \cref{thm:ForwardNoetherTheorem} is what we call a conformal symmetry of $B$. Ideally, a converse to \cref{thm:ForwardNoetherTheorem} would read as follows.

\begin{quote}
     \textbf{Converse Conformal Noether ``Theorem":} Given a flux system $(B,\nu,\mu)$, there exists a function $f$ and a vector field $u$ which is divergence-free with respect to $f\mu$ such that $\iota_u\iota_B\mu = \nu$. Consequently, $u$ is a conformal symmetry satisfying $[u,B/f] = 0$ and $\nu(u) = 0$.
\end{quote}

What \cref{thm:uexistsGivenEta} says in relation to this Converse Conformal Noether``Theorem" is that this ``Theorem" is true under the additional assumption that a 1-form adapted to $(B,\nu,\mu)$ exists. However, the naive converse Noether ``Theorem" is false: that is, the above ``Theorem" is false if one requires that $f = 1$. This is seen in \cref{counterexample} in \cref{sec:naiveNoetherCounterExample}.

\subsection{An elementary proof of Thm.~\ref{thm:uexistsGivenEta}}\label{sec:ElementaryNoether}

Here we give a constructive and elementary proof of \cref{thm:uexistsGivenEta}. To do so, we will use the following lemma. Let $M$ be a 3-manifold with (possibly empty) boundary.

\begin{lemma}\label{lem:fluxpreservingIFFuexists}
    Suppose $B$ is divergence-free with respect to the volume form $\mu$, $\eta$ a 1-form such that $\eta(B) >0$, and $\nu$ a closed 1-form. Then $(B,\nu,\mu)$ is a flux system if and only if there exists a unique vector field $u$ such that $\iota_u \iota_B \mu = \nu$ and $\eta(u) = 0$, in which case, $u$ satisfies
    \[ \iota_u\mu = \frac{\nu\wedge \eta}{\eta(B)}. \]
\end{lemma}
\begin{proof}
    The proof of this lemma is similar to that of \cite[Lemma~3.4]{burby2021integrability}. One direction is clear; if $\iota_u \iota_B \mu = \nu$ then $0 = \iota_B \iota_u \iota_B \mu = \iota_B\nu$ regardless of whether $\eta(u) = 0$. Hence, $(B,\nu,\mu)$ is a flux system. 
    
    To prove the other direction, it suffices to show that the linear map $\hat{b}: \Gamma(TM) \to \Omega^1(M)$ from vector fields to 1-forms defined by
    \[ \hat{b}(X) = \iota_X \iota_B \mu + \eta(X)\eta \]
    is a bundle isomorphism with inverse $\hat{b}^{-1} :\Omega^1(M) \to \Gamma(TM)$ defined by
    \[\iota_{\hat{b}^{-1}(\alpha)}\mu = \frac{\alpha\wedge\eta}{\eta(B)}+\frac{\alpha(B)}{\eta(B)^2}\iota_B\mu\]
    where products (interior or exterior) and evaluations (of k-forms) are interpreted are occurring at the according fibres of $TM$ and $T^*M$. Indeed, with this fact, if $u'$ is a vector field satisfying $\iota_{u'} \iota_B \mu = \nu$ and $\eta(u') = 0$, then
    \[\hat{b}(u') = \iota_{u'} \iota_B \mu + \eta(u')\eta = \nu.\]
    On the other hand, the vector field $u = \hat{b}^{-1}(\nu)$ satisfies $\nu = \iota_{u} \iota_B \mu + \eta(u)\eta$ and evaluating this equation on $B$ shows that $\eta(u) = 0$ because $\eta(B) > 0$. 
    
    Now, to establish the isomorphism, let $X \in \Gamma(TM)$ and $\alpha \in \Omega^1(M)$. Then, because $\eta(B) > 0$, the following are equivalent.
    \begin{align*}
    \hat{b}(X) = \alpha & \Longleftrightarrow \iota_X\iota_B\mu + \eta(X) \eta = \alpha\\
    & \Longleftrightarrow (\iota_X\iota_B\mu + \eta(X) \eta)(B) = \alpha(B) \text{ and } (\iota_X\iota_B\mu + \eta(X) \eta)\wedge \eta = \alpha \wedge \eta.
    \end{align*}
    On the other hand,    
    \[(\iota_X\iota_B\mu + \eta(X) \eta)(B) = \eta(X) \eta(B)\]
    and
    \begin{align*}
    (\iota_X\iota_B\mu + \eta(X) \eta)\wedge \eta &= \iota_X\iota_B\mu \wedge \eta\\
    &= \iota_X(\iota_B\mu \wedge \eta)-\eta(X)\iota_B\mu\\
    &= \eta(B)\iota_X \mu-\eta(X)\iota_B\mu.
    \end{align*}
    Hence, we have the additional equivalent statements.
    \begin{align*}
    &\hat{b}(X) = \alpha\\
    & \Longleftrightarrow \eta(X) \eta(B) = \alpha(B) \text{ and }\eta(B)\iota_X \mu-\eta(X)\iota_B\mu = \alpha \wedge \eta\\
    & \Longleftrightarrow \eta(X) = \frac{\alpha(B)}{\eta(B)} \text{ and } \iota_X \mu = \frac{\alpha \wedge \eta}{\eta(B)} + \frac{\alpha(B)}{\eta(B)^2}\iota_B\mu.
    \end{align*}
    Now, observe that
    \[\iota_X \mu \wedge \nu = \eta(X)\mu, \qquad \left(\frac{\alpha \wedge \eta}{\eta(B)} + \frac{\alpha(B)}{\eta(B)^2}\iota_B\mu\right) \wedge \eta = \frac{\alpha(B)}{\eta(B)}\mu,\]
    and therefore
    \[\iota_X \mu = \frac{\alpha \wedge \eta}{\eta(B)} + \frac{\alpha(B)}{\eta(B)^2}\iota_B\mu \implies \eta(X) = \frac{\alpha(B)}{\eta(B)}.\]
    Hence, we obtain the equivalence
    \begin{align*}
    &\hat{b}(X) = \alpha \Longleftrightarrow \iota_X \mu = \frac{\alpha \wedge \eta}{\eta(B)} + \frac{\alpha(B)}{\eta(B)^2}\iota_B\mu
    \end{align*}
    and therefore the proof of our claim.
\end{proof}

We are now ready to prove \cref{thm:uexistsGivenEta}.

\begin{proof}[Proof of \cref{thm:uexistsGivenEta}]
Let $(B,\nu,\mu)$ be a flux system on $M$. Assume that $\eta$ is an adapted 1-form for $(B,\nu,\mu)$. Then, by \cref{lem:fluxpreservingIFFuexists} the vector field $u$ given by
    \[ \iota_u\mu = \frac{\nu\wedge \eta}{\eta(B)}\]
is the unique vector field satisfying
    \[\iota_u \iota_B\mu = \nu, \qquad \eta(u) = 0.\]
It remains to show claims 1 to 4 of \cref{thm:uexistsGivenEta} on the relationship between $u$ and $\tB := \eta(B)^{-1} B$. Claims 1 and 2 can immediately be seen from $\iota_u \iota_B\mu = \nu$. Equation \eqref{eq:ConformalNoether} with $\tB$ and $\tmu = \eta(B) \mu$ shows that claims 3 and 4 are equivalent. Indeed,
\[\iota_{[u,\tB]}\tmu = d \iota_u \iota_B \mu - (\Div_{\tmu}u) \iota_B\mu = d \nu - (\Div_{\tmu}u) \iota_B\mu = - (\Div_{\tmu}u) \iota_B\mu.\]
In addition, claim 4 holds because
\[\mathcal{L}_u \tmu = d \iota_u\tmu = d(\nu \wedge \eta) = \nu \wedge d\eta = 0\]
where we used that $d\eta \wedge \nu = 0$. Thus, all claims hold, completing the proof.
\end{proof}

\begin{remark}
From the proof of \cref{lem:fluxpreservingIFFuexists}, the map $\hat{b}$ can be understood as being a vector bundle isomorphism $b : TM \to T^*M$ whereby
\[b(v) = \iota_v (\iota_B \mu)|_x + \eta|_x(v)\eta|_x \]
for $v \in TM$. This clarifies that $\hat{b}$ is a purely linear-algebraic construction.
\end{remark}


\subsection{A proof of Thm.~\ref{thm:uexistsGivenEta} through a coisotropic embedding}\label{sec:presymplecticPerspective}

In this section, theoretical context is given to \cref{thm:uexistsGivenEta}. In particular, we will take the perspective of \cite{burby2021integrability} in viewing a non-vanishing, divergence-free vector field $B$ on a 3-manifold $M$ as giving $M$ the structure of a presymplectic manifold. Through this perspective, the connection between \cref{thm:uexistsGivenEta} and the usual Noether theorem for symplectic manifolds can be understood.

\subsubsection{A brief introduction to presymplectic geometry}

The vector field $B$ and the volume form $\mu$ which it preserves gives rise to a presymplectic manifold. This change in perspective is facilitated by the observation that the 2-form 
\[ \beta = \iota_B \mu \]
is necessarily closed ($d\beta = 0$) due to the fact that $B$ is divergence-free, and that $\beta$ is of maximal rank by the non-vanishing assumption of $B$. In the plasma physics literature, the form $\beta$ is often called the \emph{flux form}.

Any closed, maximal rank 2-form $\beta$ is called \emph{presymplectic} and makes $(M,\beta)$ a \emph{presymplectic manifold}. The study of a presymplectic manifold is called presymplectic geometry. We follow \cite{burby2021integrability} in briefly summarising some key observations in presymplectic geometry. In particular, we will point out some differences to symplectic geometry arising from the fact that $\beta$ is only of maximal rank and not non-degenerate. 
\begin{definition}[Presymplectic and Hamiltonian Vector Fields]
    A vector field $u$ on $M$ is \emph{presymplectic with respect to $\beta$} or is \emph{locally Hamiltonian} if there exists a closed 1-form
    $\nu$ such that $\iota_u \beta = \nu$. The pair $(u,\nu)$ is called a \emph{presymplectic pair}. If in addition $\nu$ is exact, that is $\nu = -dH$ for some $H\in C^\infty(M)$, then $u$ is called a \emph{Hamiltonian vector field} for the Hamiltonian $H$, and $(u,H)$ is a \emph{Hamiltonian pair}.
\end{definition}

\begin{remark}
    Observe for any vector field $u$ that $\L_u \beta = \iota_u d\beta + d \iota_u \beta = d \iota_X \beta$ and hence $u$ preserves $\beta$ if and only if $u$ is presymplectic.
\end{remark}

Every Hamiltonian pair is presymplectic, and the converse is true if the first de Rham cohomology of $M$, $H^1(M)$, vanishes. It is true that, given a presymplectic pair, $B$ preserves $\nu$ as shown in the following.
\begin{lemma}[\cite{burby2021integrability}]\label{lem:presym}
    Suppose that $(u,\nu)$ is a presymplectic pair for $\beta$. Then
    \begin{enumerate}
        \item $\nu(B) = 0$;
        \item $\L_B \nu =  0$;
        \item $\beta \wedge \nu = 0 $.
    \end{enumerate}
    Moreover, if $(u,H)$ is a Hamiltonian pair, $B$ is tangent to surfaces of constant $H$, that is $\L_B H = 0$.
\end{lemma}

\cref{lem:presym} has the following natural converse which follows immediately from \cref{lem:fluxpreservingIFFuexists}.
\begin{lemma}\label{lem:PresympCriteria}
    If $\nu$ is a closed 1-form with $\nu(B) = 0$, then there exists $u$ such that $(u,\nu)$ is a presymplectic pair.
\end{lemma}

It is natural to compare symplectic vector fields to presymplectic vector fields. A symplectic 2-form $\omega$ is non-degenerate and consequently can only be defined on some even dimensional manifolds. The non-degeneracy of a symplectic form implies that, given any function $H\in C^\infty(M)$, there exists a unique Hamiltonian vector field $X_H$ defined by $\iota_{X_H} \omega= -dH$.
In the presymplectic setting, given a function $H$ on $M$, a Hamiltonian vector field need \emph{not} exist. Moreover, even if it does exist, the Hamiltonian vector field is not unique. Indeed, if $u$ is a Hamiltonian vector field for $\nu$ then any vector field of the form $u + h B$  is also Hamiltonian for $\nu$, for any function $h$ on $M$.

Perhaps more surprisingly, unlike the symplectic case, not all presymplectic vector fields are divergence-free.
\begin{lemma}[\cite{burby2021integrability}]\label{lem:divIsCommute}
    A presymplectic vector field $u$ is divergence-free if and only if $[u,B] = 0$.
\end{lemma}

The consequence of \cref{lem:PresympCriteria,lem:divIsCommute} is that one must be careful when trying to translate well-known results and theory from symplectic geometry to presymplectic geometry. In particular, as we will see in the proceeding sections, one needs additional assumptions in order to provide a Noether theorem for presymplectic vector fields, and consequently, divergence-free vector fields.

\subsubsection{Converse Conformal Noether Theorem}

In this section we establish a second proof of \cref{thm:uexistsGivenEta} from the perspective of presymplectic geometry. Before doing so, we discuss an initial approach to the Converse Conformal Noether ``Theorem" described in \cref{fwdNoether} from this perspective. 

Let $\nu$ be a closed 1-form such that $\nu(B) = 0$. Then \cref{lem:PresympCriteria} enables us to fix a vector field $u_0$ so that $(u_0,\nu)$ is a presymplectic pair: $\iota_{u_0} \beta = \nu$. Then, there exists a divergence-free vector field $u$ satisfying $\iota_u\beta = \nu$ if and only if there exists a function $h \in C^{\infty}(M)$ such that the vector field $u = u_0 + h B$ satisfies $\Div_{\mu} u = 0$ or equivalently that $h$ satisfies the cohomological equation
\[B(h) = -\Div_{\mu} u_0.\]
However, in general, only special vector fields are able to solve given cohomological equations \cite{Kocsard2009,newcombMagneticDifferentialEquations1959}. In fact, \cref{counterexample} in \cref{sec:naiveNoetherCounterExample} shows that such a $h$ need not exist. On the other hand, we have that
\[\iota_{[u,B]}\beta = \L_u \iota_B \beta - \iota_B \L_u \beta = - \iota_B d \nu = 0. \]
so $[u,B] = \kappa B$ for some function $\kappa$ and $u$ being a conformal symmetry is equivalent to another cohomological equation. This line of approach requires a deeper understanding of the vector field $B$ and the associated cohomological equation to make any conclusion.

An alternative approach comes delivered in a more abstract package. For the sake of the exposition, assume that $\nu = d\psi$ is exact. The journey begins with $\eta$-embeddings of the presymplectic manifolds.
\begin{theorem}[$\eta$-embedding {\cite{burby2021integrability}}]\label{thm:eta-embedding}
        Let $(M,\beta)$ be a presymplectic manifold. Suppose there exists a 1-form $\eta$ on $M$ such that $\beta\wedge \eta$ is a volume form. Let $\tilde{M} = M\times \R \ni (z,v)$, define the projection $\pi:\tilde{M} \to M$ and take $\tilde{\beta} = \pi^*\beta,\, \tilde{\eta} = \pi^*\eta$. Then, $(\tilde{M},\omega)$ is a symplectic manifold with
        \begin{equation*}
            \omega = \tilde{\beta} + d(v \tilde{\eta}).
        \end{equation*}
        Moreover, $M$ is coisotropically embedded as $v = 0$ in $\tilde{M}$.
\end{theorem}
Note that on a presymplectic manifold $(M,\beta)$, a metric $g$ provides the required 1-form by setting $\eta := B^\flat$. There is of course some choice in $\eta$, and we say any specific embedding is an \emph{$\eta$-embedding of $(M,\beta)$}.

The $\eta$-embedding provides a mechanism to apply the tools of symplectic geometry to that of presymplectic geometry. In order to find a conformal symmetry, we will need to not only embed the presymplectic manifold associated to $B$ into a symplectic manifold, but do so in such a way as to have the lift of $B$ to $\tilde{M}$ be part of an integrable system (in the sense of Liouville). With a little more structure imposed on $\eta$, this is indeed possible.
\begin{theorem}\label{thm:integrableEmbedding}
        Assume $\eta$ is a 1-form adapted to the flux system $(B,d\psi,\mu)$. Let $(M\times\R,\omega)$ be the symplectic manifold obtained from the $\eta$-embedding of $(M,\beta)$, take coordinates $(z,v)\in M\times\R$ and let $\tilde{\psi} = \pi^*\psi$. Then $H =  v$ and $\tilde{\psi}$ form an integrable system on $M\times\R$ with associated Hamiltonian vector fields respectively satisfying
            \begin{equation*}
                X_H|_{v=0} = \frac{1}{\eta(B)}B, \quad X_{\tilde{\psi}}|_{v=0} =  u - \frac{\eta(u)}{\eta(B)} B,
            \end{equation*}
        for any vector field $u$ such that $\iota_u \beta = -d\psi$.
\end{theorem}
\begin{proof}
    The proof is identical to that of \cite[Thm.~4.11]{burby2021integrability} which has stronger hypotheses to \cref{thm:integrableEmbedding}. It is enough to observe that the original proof follows with the weaker hypotheses. The proof is included here for completeness.
    
    By \cref{lem:PresympCriteria}, there always exists some $u$ such that $\iota_u\beta = -d\psi$, although such a $u$ is not guaranteed to be divergence-free. 
    We begin by computing the Hamiltonian vector fields associated with $H$ and $\tilde{\psi}$ on $v=0$. They are defined through Hamilton's equations:
        \[ \iota_{X_H}\omega = - dv,\quad \iota_{X_{\tilde{\psi}}}\omega = -d\tilde{\psi}. \]
    A calculation reveals that $X_H|_{v=0} = \frac{1}{\eta(B)}B$ and $X_{\tilde{\psi}}|_{v=0}=  u - \frac{\eta(u)}{\eta(B)} B$, as desired. 
    
    To show that $(H,\tilde{\psi})$ form an integrable system on $M\times\R$ we require the Poisson bracket $\{\tilde{\psi},H\} = 0$. A quick way to see this is to pull back $\omega$ to the common level set $(\tilde{\psi},H) = (\tilde{\psi}_0,H_0)$. The condition $H = H_0$ implies $v = H_0$. The pullback of $\omega$ to the level set is therefore
    \[ 
          \pi_{\tilde{\psi}_0,H_0}^* (\tilde{\beta} + d(v \tilde{\eta})) = 0 +\pi_{\tilde{\psi}_0}^*d(H_0 \tilde{\eta}) = H_0\pi_{\tilde{\psi}_0}^*d\tilde{\eta} = 0,
    \]
    since $\tilde{\beta} = \pi^*\beta$ and $d\tilde{\eta} = \pi^*d\eta$ both vanish when pulled back to a $\tilde{\psi}$-surface. It follows that $\omega$ vanishes on the common level sets, and hence, that $\{\tilde{\psi},v\} := \iota_{X_H}\iota_{X_{\tilde{\psi}}} \omega = 0$.
\end{proof}

By $\eta$-embedding the presymplectic manifold generated from the field $B$ of a flux system $(B,d\psi,\mu)$, \cref{thm:integrableEmbedding} provides an integrable system (in the sense of Liouville) given by functions $u,\tilde{\psi}$ and associated Hamiltonian vector fields $X_v, X_{\tilde{\psi}}$. As $u,\tilde{\psi}$ Poisson commute, it follows that $X_v, X_{\tilde{\psi}}$ commute. In particular the vector fields
\[ \tB := X_v|_{v=0} = \frac{1}{\eta(B)} B,\qquad \tilde{u} := X_\rho|_{u=0} = u - \frac{\eta(u)}{\eta(B)} B \]
will commute. In fact $\iota_{\tilde{u}}\beta = -d\psi$. We thus have what we want: a conformal symmetry $\tilde{u}$ of $B$. Hence, we have provided an alternative proof of \cref{thm:uexistsGivenEta} for the case $\nu$ is exact. \cref{thm:integrableEmbedding} can be adjusted to prove the case where $\nu$ is no longer exact.

\subsection{A counter-example to the converse naive Noether theorem}\label{sec:naiveNoetherCounterExample}

To understand the counter-example, one must bear in mind the following elementary results in the study of Diophantine vector fields (for a definition of Diophantine vector fields see \cite{Kocsard2009}). The first statement follows directly from \cite[Proposition 23 and 24]{perrella2022rectifiability} while the second statement is in \cite[Proposition 2.6]{Kocsard2009}.

\begin{proposition}\label{solvabilityofcohomological}
On the 2-torus $S = (\R/\mathbb{Z})^2$, the following holds.
\begin{enumerate}
    \item Let $0 < f \in C^{\infty}(S)$ and consider the vector field
    \[X = f(a\partial_x + b\partial_y)\]
    where $(a,b) \in \R^2$. If $(a,b)$ is not Diophantine, then $X$ is not smoothly conjugate to a Diophantine vector field. That is, there does not exist a diffeomorphism $\varphi : S \to S$ such that    
    \[\varphi_* X = c\partial_x + d\partial_y\]
    for any Diophantine vector $(c,d) \in \R^2$.
    \item A vector field $X$ on $S$ is smoothly conjugate to a Diophantine vector field, if and only if, for every $h \in C^{\infty}(S)$, there exists a $g \in C^{\infty}(S)$ and a constant $c \in \R$ solving the cohomological equation
    \[X(g) = h-c.\]
\end{enumerate}
\end{proposition}

With this in mind, we now consider the example.

\begin{example}\label{counterexample}
First consider the 2-torus $S = (\R/\mathbb{Z})^2$. With the natural coordinates $x,y$, consider the vector fields
\[X_0 = a \partial_x + b\partial_y, \qquad Y_0 = c \partial_x + d\partial_y\]
where $(a,b) \in \R^2$ be linearly independent from $\mathbb{Z}^2$ but not Diophantine and $(c,d)$ is Diophantine and linearly independent from $(a,b)$. We will show that for some $0<f \in C^{\infty}(S)$, the only vector fields commuting with $X = fX_0$ are constant multiples of $X$. 

For the time being, let $f \in C^{\infty}(S)$ be arbitrary and set $X = fX_0$. Any vector field $Z$ may be written, for some $g,h \in C^{\infty}(S)$, as
\[Z = g X + h Y_0.\]
In this way, we obtain
\begin{align*}
[X,Z] &= [X,gX+hY_0]\\
&= [X,gX]+[fX_0,hY_0]\\
&= (f X(g)-hY_0(f))X_0 + f X_0(h)Y_0. 
\end{align*}
Hence, if $[X,Z] = 0$, then $X_0(h) = 0$ and since $X_0$ has dense orbits in $S$, $h$ must be constant. Moreover, we have the cohomological equation
\begin{equation}\label{eq:cohomexample}
X(g) = Y_0(h \ln f)
\end{equation}
Applying the first and second statement of \cref{solvabilityofcohomological} to $X$, there exists $H \in C^{\infty}(S)$ such that, for any function $g \in C^{\infty}$ and $c \in \R$ 
\[X(g) \neq H-c.\]
Fix such a $H$. Suppose that $h \neq 0$. Applying the second statement of \cref{solvabilityofcohomological} to $Y_0$, there exists function $f$ and a constant $c'$ such that
\[Y_0(h \ln f) = H -c'.\]
Fixing this choice of $f$, \eqref{eq:cohomexample} is false. Therefore, we must have that $h=0$ and so $X(g) = 0$. This implies as before that $g$ is constant and $Z = g X$, a constant multiple of $X$. Hence, our claim is true.

With this in mind, we now consider the manifold with (possibly empty) boundary $M = S \times I$ where $I$ is some interval. With the natural coordinates $x,y,z$, consider the flux system $(B,\nu,\mu)$ where
\[B = F(a \partial_x + b\partial_y), \qquad \nu = dz, \qquad \mu = F^{-1}dx\wedge dy \wedge dz\]
and $F(x,y,z) = f(x,y)$ where $f$ is as before. Then, the previous paragraph shows that $B$ is not rectifiable, not even locally about a single constant-$z$ torus. In particular, there cannot exist a vector field $u$ linearly independent from $B$ satisfying
\[[u,B] = 0,\qquad \nu(u) = 0\]
let alone satisfying $\iota_u\iota_B\mu = \nu$. 
\end{example}

\subsection{Contexts in which an adapted 1-form is guaranteed globally} \label{sec:whenEtaIsGuaranteed}

With some additional assumptions on a flux system, an adapted 1-form can be guaranteed to exist globally. Fix a 3-manifold $M$ with (possibly empty) boundary.

\begin{proposition}\label{prop:etaExists}
Let $B$ be a vector field on $M$. Let $\eta$ be a 1-form such that
    \[ \iota_B d\eta = 0, \qquad  \eta(B) > 0.\]
If $(B,\nu,\mu)$ is a flux system, then $\eta$ is a 1-form adapted to $(B,\nu,\mu)$.
\end{proposition}
\begin{proof}
    It suffices to note that $d\eta \wedge \nu$ is a top-form on $M$ and hence $d\eta\wedge \nu=0$ is equivalent to $\iota_B(d\eta \wedge \nu) = 0$. Moreover,
    \[ \iota_B(d\eta \wedge \nu)  =  (\iota_B d\eta) \wedge \nu + d\eta \wedge \iota_B\nu = 0\]
\end{proof}

There are some particular cases of \cref{prop:etaExists} of interest. If $B$ is vector field and $\eta$ is a closed 1-form such that $\eta(B) > 0$, then the hypothesis of \cref{prop:etaExists} are satisfied. If $M$ is closed, any vector field $B$ with a global Poincar\'{e} section possesses such an $\eta$ by the work of \cite{SchwartzmanGlobal}. In plasma physics, magnetic fields $B$ in a vacuum on an oriented $M$ satisfy $d B^\flat = 0$ with respect to the metric $g$ placed on $M$. Hence, all non-vanishing vacuum fields have the closed 1-form $\eta = B^\flat$ satisfying the assumptions of \cref{prop:etaExists}.

Alternatively, if $B$ is a vector field and $\eta$ is a 1-form such that $d \eta$ is non-vanishing, then $B$ and $\eta$ satisfy the hypothesis of \cref{prop:etaExists} if and only if $\eta$ is a contact form and $B$ is a Reeb-like vector field of $\eta$. In this context, $B$ is a Reeb-like field if there exists a non-vanishing function $f \in C^\infty(M)$ such that $B/f$ is the Reeb vector field of $\eta$. For such fields, $B$ is divergence-free with respect to $\mu = \frac{1}{\eta(B)}\eta \wedge d\eta$. The following is an example of a contact system giving rise to an adapted 1-form.

\begin{example}\label{ex:T3Twist}
    Set $M = (\R / 2\pi \mathbb{Z})^3$ with periodic coordinates $x,y,z : (\R/ 2\pi \mathbb{Z})^3 \to \R/ 2\pi \mathbb{Z}$. Consider the contact form
    \begin{equation*}
        \eta = \sin z dx + \cos z dy    
    \end{equation*}
    whereby $\eta \wedge d\eta = dx\wedge dy\wedge dz$, the standard volume form on $M$. The associated Reeb vector field is given by
    \begin{align*}
        B &= \sin z \partial_x + \cos z \partial_y.
    \end{align*}
    With $\nu = dz$, we get a flux system $(B,\nu,\mu)$ and $\eta$ an adapted 1-form.
\end{example}

It is worth highlighting that \cref{ex:T3Twist} gives an example of a flux system such that $B$ has no global Poincar{\'e} section, and $\eta$ is not closed. In fact, the adapted 1-forms for the flux systems which arise from contact forms can never be closed in the case that $M$ is closed. This is because of the following elementary fact.
\begin{proposition}
Let $B$ be a Reeb-like vector field for a contact form $\eta$ on a manifold $M$. Then, $B$ has no closed Poincar{\'e} section. In particular, if $M$ is closed, $B$ has no closed 1-form $\kappa$ for which $\kappa(B) > 0$.
\end{proposition}

\begin{proof}
Let $S$ be a closed hypersurface of dimension $2n$ of $M$ where $\dim M = 2n + 1$ ($n \geq 1$). Considering the volume-form $\mu = \eta \wedge (d\eta)^n$ and the Reeb vector field $R$, we have $\iota_R \mu = (d\eta)^n$. Then considering the inclusion $i : S \subset M$, because $i^*(d\eta)^n$ is exact on $S$, Stokes' Theorem gives
\begin{equation*}
\int_S i^*\iota_R\mu = \int_S i^*(d\eta)^n = 0.
\end{equation*}
Hence, $R$ is not transverse to $S$. Thus, $B$ is not transverse to $S$ either. In particular, $B$ has no Poincar{\'e} section. Moreover, it follows from the same techniques as in \cite{Tischler1970} that a vector field $X$ possesses a closed transverse section provided there exists a closed 1-form $\kappa$ such that $\kappa(X) > 0$. This concludes the proof.
\end{proof}

As an example with relevance to applications, Beltrami fields can be thought of as generalisations of Reeb-like vector fields in 3 dimensions \cite{etnyreContactTopologyHydrodynamics2000}. If $M$ has a metric and an orientation, a Beltrami field $B$ satisfies
\begin{equation*}
\nabla \cdot B = 0, \qquad \nabla \times B = \lambda B    
\end{equation*}
for some $\lambda \in C^{\infty}(M)$. In terms of $\beta = \iota_B\mu$ where $\mu$ is the induced volume form, this equation reads
\begin{equation*}
d B^{\flat} = \lambda \beta    
\end{equation*}
and therefore $\eta = B^{\flat}$ is a 1-form where the assumptions of \cref{prop:etaExists} are satisfied provided that $B$ is non-vanishing.


\section{An adapted 1-form exists in the toroidal region}\label{sec:eta}

The main purpose of this section is to prove \cref{thm:etaexists}. The proof for a given tangential flux system $(B,\nu,\mu)$ consists of two steps. In step one, the local existence of closed 1-forms adapted to the flux system is established in a neighborhood of each of the torus and axis in the toroidal region. Step one is achieved using tubular neighborhoods and in the case of the tori, we use flowouts which are essentially the restriction of a transverse flow to the torus; these are also used in the second step. In step two, we patch together these closed 1-forms using a partition of unity. In step two, care is taken to ensure that the functions in the partition of unity have differentials that are linearly dependent with $\nu$. The construction of the partition relies on techniques which also prove \cref{classification}. These techniques apply more generally to closed 1-forms in a region which we call the closed region. This generalisation is adopted to clarify the role that each element of the flux system plays in the proof and because it does not introduce any additional difficulties to the proof.

This section proceeds as follows. In \cref{sec:toroidal}, the toroidal region is formally defined and some auxiliary definitions used to prove \cref{thm:etaexists} are introduced. Then, in Section \ref{sec:proofthatetaexists}, \cref{thm:etaexists} is proven under the assumption that the suitable partition of unity has been constructed. Lastly, in Section \cref{partitionandclassification}, the partition of unity is constructed and \cref{classification} is proven.

\subsection{The toroidal region and the closed region}\label{sec:toroidal}

We begin by introducing some convenient notation used throughout this section. If $M$ is a manifold with (possibly empty) boundary, $S$ is a submanifold, and $i : S \subset M$ is the inclusion, then for any $k$-form $\nu$ on $M$, we denote $S^*\nu = i^*\nu$. 

The toroidal region arises from a more general concept which we call the closed region. The closed region accounts for the simplest kind of critical sets which can occur for a closed 1-form $\nu$ on $M$ such that $\partial M^*\nu = 0$ and will now be defined.

\begin{definition}\label{closeddef}
Let $\Omega = \{x \in M : \nu|_x \neq 0\}$ be the set of non-critical points of $\nu$. The boundary condition $\partial M^*\nu = 0$ implies that the non-vanishing closed 1-form $\Omega^*\nu$ defines an integrable distribution $D$ on $\Omega$ regarded as an open submanifold with boundary of $M$ (see \cref{app:distandfol} for a formal review of basic distribution and foliation theory on manifolds with boundary). Let $\mathcal{F}$ be the associated foliation on $M$ to $D$.
\begin{enumerate}
    \item A closed leaf $S$ of the induced foliation is a \emph{regular closed leaf of $\nu$}. The \emph{pre-closed region of $\nu$}, denoted $\mathcal{C}^2$, is the union of all regular closed leaves of $\nu$
    \item A \emph{degenerate closed leaf of $\nu$} is a connected closed submanifold $K$ of codimension at least $2$ which is contained in $\interior M$ and satisfying $\nu|_K = 0$ and possessing a neighbourhood $U$ in $M$ such that $U \backslash K \subset \mathcal{C}^2$.
    \item Set $\mathcal{C}^1$ to be the union of degenerate closed leaves of $\nu$. Then, the closed region $\mathcal{C}$ is the union $\mathcal{C}^1 \cup \mathcal{C}^2$.
\end{enumerate}
\end{definition}

\begin{remark}[on \cref{closeddef}] The regular closed leaves of $\nu$ coincide with the compact integral manifold of $D$, since integral manifolds of a distribution which are closed as subsets of the ambient space are necessarily maximal.
\end{remark}

We now define the toroidal region.

\begin{definition}\label{toroidaldef}
Assume that $\dim M = 3$ and let $(B,\nu,\mu)$ be a tangential flux system on $M$ (see \cref{def:fluxsystem}).
\begin{enumerate}
    \item A \emph{regular torus of $(B,\nu,\mu)$} is a regular closed leaf $S$ of $\nu$ such that $B|_S$ is non-vanishing. The \emph{pre-toroidal region of $(B,\nu,\mu)$}, denoted $\mathcal{T}^2$, is the union of all regular tori.
    \item An \emph{axis of $(B,\nu,\mu)$} is a 1-dimensional periodic orbit $\gamma$ of $B$ which is also a degenerate closed leaf of $\nu$.
    \item Set $\mathcal{T}^1$ to be the union of all axes. The \emph{toroidal region of $(B,\nu,\mu)$} is then the subset $\mathcal{T} = \mathcal{T}^1 \cup \mathcal{T}^2$.
\end{enumerate}
\end{definition}

\begin{remark}[on \cref{toroidaldef}]
It is well-known from index theory that the only compact orientable surface admitting a nowhere zero vector field is the 2-torus. In particular, the regular tori are indeed tori. 
\end{remark}

\subsection{Proof of Thm.~\ref{thm:etaexists}}\label{sec:proofthatetaexists}

Our main tool for proving \cref{thm:etaexists} is a flowout. This terminology is adopted from \cite{Lee2012} and for our purposes means the following. Before using the specific assumptions in \cref{thm:etaexists}, we will focus more generally on the closed region because of its relevance to later sections. Let $M$ be a manifold with boundary and $\nu$ a closed 1-form on $M$ such that $\partial M^*\nu = 0$. Let $\Omega = \{x \in M : \nu|_x = 0\}$ and fix a vector field $N$ on $\Omega$ such that $\nu(N) = 1$.

\begin{definition}
For any $S$ a closed leaf of $\nu$ and $I$ an interval containing $0$, if for each $x \in S$, the integral curve $\gamma_x$ of $N$ starting at $x$ is defined for all $t \in I$, then the smooth map $\Sigma : S \times I \to M$ such that $\Sigma(x,t) = \gamma_{x}(t)$ is called the \emph{flowout of $S$ with interval $I$}.
\end{definition}

The following establishes elementary facts about of flowouts we require for proving \cref{thm:etaexists}.

\begin{lemma}\label{flowoutexistence}
Let $S$ be a closed leaf of $\nu$. Then either $S \subset \interior M$ or $S$ is a connected component of $\partial M$. Moreover, there exists $\epsilon > 0$ and an interval $I$ such that one of the following holds.
\begin{enumerate}
    \item if $S \subset \interior M$, then $I = (-\epsilon,\epsilon)$,
    \item and if $S$ is a component of $\partial M$ and $N$ is inward-pointing on $S \subset \partial M$, then $I = [0,\epsilon)$ and otherwise $I = (-\epsilon,0]$,
\end{enumerate}
and a flowout $\Sigma$ with interval $I$ which is a diffeomorphism onto its open image in $M$.
\end{lemma}

\begin{proof}
Note that the boundary $\partial \Omega$ is closed in $\Omega$. Suppose that $S \cap \partial M \neq \emptyset$. Then, we may fix a point $x \in S \cap \partial \Omega$. Consider now the connected component $C$ of $\partial \Omega$ containing $x$. Because $\partial \Omega$ is embedded in $\Omega$, $C$ is a closed subset of $\partial \Omega$. On the other hand, it is also an integral manifold containing $x$ of the distribution $D$ since $\partial M^*\nu = 0$. In summary, by closedness, $C$ and $S$ are maximal integral manifolds containing $x$. Thus, we must have that $C = S$ from above. In particular, $S$ is a closed, open and connected subset of $\partial M$. Thus, $S$ is a connected component of $\partial M$.

The existence of flowouts now follows directly from the flowout theorems in \cite[Theorem 9.20 and Theorem 9.24]{Lee2012}.
\end{proof}

To prove openness of the toroidal region (and closed region) along with some algebraic aspects of the proof of \cref{thm:etaexists}, we will also make use of the following properties of flowouts.

\begin{lemma}\label{flowoutproperties}
Let $\Sigma$ be a flowout of $S$ with some interval $I$ containing $0$. Then the following holds.
\begin{enumerate}
    \item If $t : S \times I \to I \subset \R$ denotes projection followed by inclusion into $\R$, then $\Sigma^*\nu = dt$.
    \item If $\Sigma'$ is another flowout of $S$ with some interval $J$ containing $0$, then $\Sigma$ and $\Sigma'$ coincide on $S \times (I \cap J)$.
    \item For all $s \in I$, $\Sigma(\cdot,s) : S \to M$ is a diffeomorphism onto its image and its image $S' = \Sigma(S \times \{s\})$ is a closed leaf of $\nu$. Moreover, the flowout $\Sigma' : S' \times I' \to M$ of $S'$ with interval $I' = I - s$ exists and satisfies, for $x \in S$ and $t' \in I'$,
    \begin{equation*}
    \Sigma'(\Sigma(x,s),t') = \Sigma(x,s + t').
    \end{equation*}
\end{enumerate}
\end{lemma}

\begin{proof}
Concerning the first statement, set $\kappa =\Sigma^*\nu$. Let $V$ be a contractible open subset of $S$. Then, we have the contractible open subset of $V \times I$ so that $(V \times I)^*\kappa = df$ for some $f \in C^{\infty}(V \times I)$. Now, let $p \in V$ and consider the curve $\gamma_p : I \to U$ given by $\gamma_p(t) = (p,t)$. Then, because $\kappa(\partial_t) = 1$, we have
\begin{align*}
(f \circ \gamma_p)'(t) = df(\gamma_p'(t)) = 1.
\end{align*}
Hence, $f \circ \gamma_p = t+f(p,0)$. On the other hand because $(S \times \{0\})^*\kappa = 0$, we have that $f(\cdot,0) = c$ for some $c \in \R$. Thus
\begin{equation*}
f = t+c.    
\end{equation*}
Hence $\kappa|_{V \times J} = dt|_{V \times J}$. Since $V$ was arbitrary, we conclude that $\Sigma^*\nu = \kappa = dt$. The second statement follows directly from uniqueness of integral curves. Having established these prior statements, the third statement now follows directly from the Fundamental Theorem of Flows in \cite[Theorem 9.12]{Lee2012}.
\end{proof}

We now show the openness of the closed and toroidal regions and verify that the critical points of $\nu$ are nowhere dense in these regions.

\begin{lemma}\label{regionisopen}
The regions $\mathcal{C}^2$ and $\mathcal{C}$ of $\nu$ are open in $M$ and the critical points of $\nu$ in $\mathcal{C}$ are nowhere dense in $\mathcal{C}$. Moreover, if $(B,\nu,\mu)$ is a tangential flux system on $M$ (so that $\dim M = 3$), then the pre-toroidal $\mathcal{T}^2$ and toroidal region $\mathcal{T}$ are open in $M$.
\end{lemma}

\begin{proof}
First, if $S$ is a closed leaf of $\nu$, by Lemma \ref{flowoutexistence}, we can take a flowout $\Sigma : S \times I \to M$ of $S$ with some interval $I$ which is a diffeomorphism onto its open image $U = \Sigma(S \times I)$ in $M$. By Lemma \ref{flowoutproperties}, $U \subset \mathcal{C}^2$. Regarding $(B,\nu,\mu)$, if $B$ is non-vanishing on $S$, then $\{x \in U : B|_x \neq 0\} = \Sigma(V)$ for some neighbourhood $V$ of $S \times \{0\}$ in $S \times I$. Because $S \times \{0\}$ is compact in $S\times I$, it is elementary topology that $V$ contains $S \times $ for some open subset $J \subset I$. Thus, $\{x \in U : B|_x \neq 0\} = \Sigma(S \times J) \subset \mathcal{T}^2$. In summary, $\mathcal{C}^2$ is open and if $(B,\nu,\mu)$ is a tangential flux system, $\mathcal{T}^2$ is open in $M$. It now follows by construction that $\mathcal{C}$ is open in $M$ and that $\mathcal{T}$ are open in $M$ if $(B,\nu,\mu)$ is a tangential flux system on $M$.

Concerning critical points, if $U$ is an open set in $M$ contained in $\mathcal{C}$ with $\nu|_U = 0$, then $U \subset \mathcal{C}^1$. If $K$ is a degenerate closed leaf of $\nu$, then taking an open set $V \supset K$ so that $V\backslash K \subset \mathcal{C}^2$, we have $U\cap V \subset K$. Since each degenerate closed leaf has nonzero codimension, we must have that $U\cap K = \emptyset$. Hence, $U = \emptyset$. As required.
\end{proof}

\begin{remark}
The relationship we are exploiting between the closed region and toroidal region for a tangential flux system $(B,\nu,\mu)$ can now be clarified. Clearly we have the relationship $\mathcal{T} \subset \mathcal{C}$. But, since $\mathcal{T}$ is an open submanifold of $M$, we may consider the closed region $\mathcal{C}_{\mathcal{T}}$ of the closed 1-form $\mathcal{T}^*\nu$ which satisfies $(\partial \mathcal{T})^*(\mathcal{T}^*\nu) = 0$ because $\partial M^*\nu = 0$. It is clear that we have equality $\mathcal{C}_{\mathcal{T}} = \mathcal{T}$. Thus, statements applying to $\nu$ on $\mathcal{C}$ also apply to $\nu$ on $\mathcal{T}$ (such as the nowhere denseness of critial points of $\nu$).
\end{remark}

We now establish the local existence an adapted 1-form in a neighbourhood of each regular torus of a tangential flux system.

\begin{lemma}\label{etaontubenbhd}
Let $(B,\nu,\mu)$ be a tangential flux system on $M$. Let $S$ be a regular torus. Then, there exists a neighbourhood $U$ of $S$ in $\mathcal{T}$ and a closed 1-form $\eta \in \Omega^1(U)$ such that $\eta(B) > 0$.
\end{lemma}

\begin{proof}
By \cref{flowoutexistence} and \cref{flowoutproperties} we may take an interval $I \subset \R$ and an embedding $\Sigma : S \times I \to M$ such that $\Sigma^* \nu = dt$ where $t : S \times I \to \R$ is projection onto $I \subset \R$. Now, the map $\Sigma$ induces a volume form $\tmu = \Sigma^*\mu \in \Omega^3(S \times I)$ and vector field $\tB$ on $S \times I$ which is $\Sigma$-related to $B$. In particular, we have the relations
\begin{align*}
\mathcal{L}_{\tB}\tmu &= 0, &  dt(\tB) &= 0.
\end{align*}
On the other hand, with the vector field $\partial_t$ on $S \times I$, observe that
\begin{equation*}
\iota_{\partial_t} (dt \wedge \iota_{\partial_t} \tmu) = (\iota_{\partial_t} dt)\iota_{\partial_t}\tmu - dt \wedge \iota_{\partial_t}\iota_{\partial_t}\tmu = \iota_{\partial_t}\tmu.
\end{equation*}
Thus, $\tmu = dt \wedge \iota_{\partial_t} \tmu$ so that
\begin{equation*}
0 = \mathcal{L}_{\tB} (dt \wedge \iota_{\partial_t} \tmu) = (\mathcal{L}_{\tB} dt) \wedge \iota_{\partial_t} \tmu + dt \wedge \mathcal{L}_{\tB}\iota_{\partial_t}\tmu = dt \wedge \mathcal{L}_{\tB}\iota_{\partial_t}\tmu.
\end{equation*}
Thus with the induced vector field $B_0$ on $S_0 = S\times \{0\}$ from $\tB$ and area form $\mu_0 = (S_0)^*\iota_{\partial_t}\tmu$ we have that $\mathcal{L}_{B_0}\mu_0 = 0$. On the other hand, it is well known in general (see for instance \cite{perrella2022rectifiability}) that any area-preserving vector field $X$ on a 2-torus possess a closed 1-form $\eta_0$ such that $\eta_0(X) > 0$. In particular, fixing such an $\eta_0$ for $B_0$, via the projection $\pi : S \times I \to S_0$, we obtain a closed 1-form $\tilde{\eta} = \pi^*\eta_0$ on $S \times I$ such that $\tilde{\eta}(\tB)|_S > 0$. Moreover, the set $\tilde{U} = \{x \in S \times I : \tilde{\eta}(\tB) > 0\}$ is an open neighbourhood of $S_0$ in $S \times I$. The result now easily follows from pushing forward $\tilde{\eta}$ into $M$ with $\Sigma$.
\end{proof}

We now establish an adapted 1-form in a neighbourhood of each axis of a tangential flux system. This follows immediately from the following well-known result.

\begin{proposition}\label{prop:etaonarbaxis}
Let $X$ be a vector field on a manifold $Y$ and $\gamma$ a 1-dimensional periodic orbit of $X$. Then, there exists a neighborhood $U$ of $\gamma$ and a closed 1-form $\eta \in \Omega^1(U)$ such that $\eta(X) > 0$.
\end{proposition}

\begin{proof}
By periodicity, $\gamma$ must be an embedded 1-dimensional submanifold. Thus, by the Tubular Neighbourhood Theorem there exists a neighbourhood $U'$ of $\gamma$ and a retraction $r : U' \to \gamma$. We also have the induced vector field $\tilde{X}$ on $\gamma$ as $X$ is tangent to $\gamma$. Hence, we have a 1-form $\tilde{\eta} \in \Omega^1(\gamma)$ with $\tilde{\eta}(\tilde{X}) > 0$. Moreover, $\tilde{\eta}$ is closed as a top-form on $\gamma$. Now, consider the closed 1-form $\eta = r^*\tilde{\eta} \in \Omega^1(U')$. Since $r$ is a retraction, we have $\eta(X) > 0$ on $\gamma$. Thus, $U = \{x \in U' : \eta(X)|_x > 0\}$ is the required neighbourhood.
\end{proof}

As a corollary, we obtain the following lemma.

\begin{lemma}\label{etaonaxisnbhd}
Let $(B,\nu,\mu)$ be a tangential flux system on $M$. Let $\gamma$ be an axis. Then, there exists a neighbourhood $U$ of $\gamma$ in $\mathcal{T}$ and a closed 1-form $\eta \in \Omega^1(U)$ such that $\eta(B) > 0$.
\end{lemma}

To patch together the local adapted 1-forms around regular tori and axes, we require an adapted partition of unity which we now introduce. This may be done on the closed region. Let $F$ denote the set of leaves and degenerate closed leaves in $\mathcal{C}$.
\begin{definition}
Let $\{U_{L}\}_{L \in F}$ be an open cover of $\mathcal{C}$ such that $U_{L} \supset L$ for all $L \in F$. A \emph{partition of unity adapted to $\nu$ on $\mathcal{C}$ and subordinated to $\{U_{L}\}_{L \in F}$} is a family $\{\psi_L : \mathcal{C} \to \R\}_{L \in A}$ of smooth compactly supported non-negative functions such that
\begin{enumerate}
    \item the collection of supports $\{\text{supp}(\psi_{L})\}_{L \in F}$ is locally finite and $\text{supp}(\psi_{L}) \subset U_{L}$ for each $L \in F$,
    \item we have $\sum_{L \in F} \psi_{L} = 1$, and
    \item for all $L \in F$, $d\psi_{L} \wedge \nu = 0$.
\end{enumerate}
\end{definition}

We have the following existence of partitions of unity in the above sense. The proof is given in Section \ref{partitionandclassification}.

\begin{proposition}\label{partition of unity}
Let $\{U_{L}\}_{L \in F}$ be an open cover of $\mathcal{C}$ such that $U_{L} \supset L$ for all $L \in F$. Then, there exists a partition of unity $\{\psi_L : \mathcal{C} \to \R\}_{L \in A}$ adapted to $\nu$ on $\mathcal{C}$ and subordinated to $\{U_{L}\}_{L \in F}$.
\end{proposition}

With this, we may prove \cref{thm:etaexists}.

\begin{proof}[Proof of \cref{thm:etaexists}]
Let $F$ denote the set of all axes and regular tori. By Lemmas \ref{etaontubenbhd} and \ref{etaonaxisnbhd}, for each $L \in F$ we may choose a neighbourhood $U_{L}$ of $L$ in $\mathcal{T}$ and a closed 1-form $\eta_{L} \in \Omega^1(U_{L})$ such that $\eta_{L}(B)>0$. Then, since $F$ covers $\mathcal{T}$, the indexed set $(U_{L})_{L \in F}$ is an open cover of $\mathcal{T}$. By \cref{partition of unity}, there exists a partition of unity $\{\psi_{L} : \mathcal{T} \to \R\}_{L \in F}$ subordinated to $\{U_{L}\}_{L \in F}$ and adapted to $\nu$. Define
\begin{equation*}
\eta = \sum_{L \in F} \psi_{L}\eta_{L}    
\end{equation*}
Then, because the $\psi_{L}$ are each non-negative and $\sum_{L \in F} \psi_{L} = 1$, we have
\begin{equation*}
\eta(B) > 0.    
\end{equation*}
Lastly, we observe that
\begin{equation*}
d\eta = \sum_{L \in F} d\psi_{L}\wedge \eta_{L}     
\end{equation*}
so that, because $d\psi_{L} \wedge \nu = 0$ for each $L \in F$, we have
\begin{equation*}
d\eta \wedge \nu = 0.
\end{equation*}
\end{proof}

\subsection{Construction of the partition of unity and classification of the toroidal region}\label{partitionandclassification}

Here we will construct the partition of unity and classify the closed region $\mathcal{C}$ (thereby proving Proposition \ref{classification}). In Section \ref{step1} we establish the partition of unity for the case of a toroidal region without boundary or axes. Then, in Section \ref{step2}, we will extend the result of \ref{step1} to include the situation where there are axes and boundary points.

\subsubsection{Without boundary or axes}\label{step1}
Let $M$ be a manifold without boundary of dimension $n$ and $\nu$ be a closed 1-form on $M$. Here, we will assume that the region $\mathcal{C}^2$ coincides with $M$ (so that $\mathcal{C} = \mathcal{C}^2 = M$). Denote by $\mathcal{F}$ the rank $n-1$ foliation by $\nu$ (consisting of all the regular closed leaves). The following shows that non-injectivity of a flowout implies compactness.

\begin{lemma}\label{compactalternative}
Assume $M$ is connected and there exists a non-injective flowout $\Sigma$ of some $S \in \mathcal{F}$ for some open interval $I$ containing $0$. Then, $M$ is compact.
\end{lemma}

\begin{proof}
Because $\Sigma$ is not injective but $\Sigma(\cdot,t) : S \to M$ is injective for each $t \in I$, there exists points $(x,s), (y,t) \in S \times I$ with $\Sigma(x,s) = \Sigma(y,t)$ and $s < t$. Now by \cref{flowoutproperties}, we have a flowout $\Sigma'$ of $S' = \Sigma(S,s) \in \mathcal{F}$ with interval $I' = I - s$. With the point $x' = \Sigma(x,s) \in S'$ and time $0 < t' = t-s \in I'$ we get
\begin{equation*}
\Sigma'(x',t') = \Sigma'(\Sigma(x,s),t') = \Sigma(x,t). 
\end{equation*}
On the other hand, since $x' = \Sigma(y,t)$ we have $\Sigma(y,t) \in S'$. So that $\Sigma(x,t) \in S'$ and hence $\Sigma'(x',t') \in S'$. Altogether, we conclude so far that $\Sigma'(S', t') = S'$.

Now, consider the compact subset
\begin{equation*}
K = \Sigma'(S' \times [0,t']) \subset M.
\end{equation*}
We show that $K$ is open. Because $\Sigma'$ is an open map by \cref{flowoutproperties}, it suffices to find $\epsilon > 0$ so that $K = \Sigma'(S' \times (-\epsilon,t'+\epsilon))$. To this end, we consider the flowout $\Sigma''$ of $S' = \Sigma'(S',t')$ with interval $I'' = I'-t'$ and make the following periodicity argument. On the one hand, for $z \in S'$ and $t'' \in I''$, by \cref{flowoutproperties} we have the identity 
\begin{equation*}
\Sigma''(\Sigma'(z,t'),t'') = \Sigma'(z,t'+t'').
\end{equation*}
Since $\Sigma'(S' \times \{t'\}) = S'$, we have in particular that, for $t'' \in I''$,
\begin{equation*}
\Sigma''(S' \times \{t''\}) = \Sigma'(S' \times \{t'+t''\}).    
\end{equation*}
On the other hand, $I = I' \cap I''$ is an open interval containing $0$ and by \cref{flowoutproperties}, for each $z \in S'$ and $r \in I$, we have
\begin{equation*}
\Sigma''(z,r) = \Sigma'(z,r).     
\end{equation*}
Thus for each $r \in I$, 
\begin{equation*}
\Sigma''(S' \times \{r\}) = \Sigma'(S' \times \{r\}).     
\end{equation*}
Now, for $\epsilon > 0$ sufficiently small we have that $[0,\epsilon), (t'-\epsilon,t'] \subset [0,t']$ and $\pm \epsilon \in I$ so that the above relations imply
\begin{align*}
\Sigma'(S' \times [t',t'+\epsilon)) = \Sigma''(S' \times [0,\epsilon)) = \Sigma'(S' \times [0,\epsilon)) \subset K,\\
\Sigma'(S' \times (-\epsilon,0]) = \Sigma''(S' \times (-\epsilon,0]) = \Sigma'(S' \times (t'-\epsilon,t']) \subset K.
\end{align*}
This shows that
\begin{equation*}
K = \Sigma'(S' \times (-\epsilon,t'+\epsilon)). 
\end{equation*}
Thus, $K$ is open, as claimed. To conclude, $M$ being connected implies $M = K$, a compact subset of $M$. That is, $M$ is compact.
\end{proof}

With \cref{compactalternative} the following useful lemma can be proved. The lemma will be used substantially in \cref{partitionandclassification} due to the convenient dichotomy it gives rise to.

\begin{lemma}\label{boundarylessfoundation}
Assume that $M$ is connected. Then, one of the following holds.
\begin{enumerate}
    \item For any $S \in \mathcal{F}$, there exists a flowout $\Sigma$ of $S$ along $N$ with an interval $I$ which is open and contains $0$ such that $\Sigma : S \times I \to M$ is a diffeomorphism.
    \item $M$ is compact.
\end{enumerate}
\end{lemma}

\begin{proof}
Let $S \in \mathcal{F}$. Then, consider the subset $U$ of points $x \in M$ for which there exists an open interval $I$ containing $0$ and a flowout $\Sigma$ of $S$ with interval $I$ such that $x \in \Sigma(S\times I)$. This set $U$ is open because each flowout $\Sigma$ of $S$ with any open interval $I$ containing $0$ is an open map by \cref{flowoutproperties} and $U$ a union of images $\Sigma(S\times I)$ of such maps.

We now show that $U$ is closed. Let $y \in M$ and consider the unique $S' \in \mathcal{F}$ containing $y$. Then, let $\Sigma' : S' \times I' \to M$ be a flowout of $S'$ with some interval $I'$ containing $0$ and consider $\Sigma'(S' \times I')$. Suppose that $\Sigma'(S' \times I') \cap U \neq \emptyset$. Then, there exists a flowout $\Sigma$ of $S$ with interval $I$ such that $\Sigma(S\times I) \cap \Sigma'(S' \times I') \neq \emptyset$. Then, there must exist $t \in I$ and $t' \in I'$ such that $\Sigma(S \times \{t\}) = \Sigma'(S' \times \{t'\})$. Consider now the flowout $\Sigma''$ of $S'' = \Sigma(S \times \{t\})$ with interval $I'' = I-t'$. Recall $\Sigma''$ has the identity, for $x' \in S'$ and $t'' \in I''$,
\begin{equation*}
\Sigma''(\Sigma'(x',t'),t'') = \Sigma'(x',t'+t'').
\end{equation*}
Now, consider the open interval $\tilde{I} = I \cup (I''+t)$ and the map $\tilde{\Sigma} : S \times \tilde{I} \to M$ given by
\begin{equation*}
\tilde{\Sigma}(x,s)= 
\begin{cases}
    \Sigma(x,s) , & \text{if } s \leq t\\
    \Sigma''(\Sigma(x,t),s-t) & \text{otherwise}.
\end{cases}
\end{equation*}
It is clear that $\tilde{\Sigma}$ is a flowout of $S$ with interval $\tilde{I}$. Lastly,
\begin{equation*}
\Sigma'(S' \times I') \subset \tilde{\Sigma}(S \times \tilde{I}) \subset U.
\end{equation*}
Thus, if $y \in M \backslash U$, then $\Sigma'(S' \times I') \subset M \backslash U$. Since $\Sigma'$ is an open map and $y \in \Sigma'(S' \times I')$, this is a neighbourhood of $y$ contained in $M \backslash U$. Thus, $U$ is closed as claimed. Lastly, since $U \neq \emptyset$ and $M$ is connected, we must have that $U = M$.

Thus, for every $y \in M$, we may choose an open interval $I_y$ containing $0$ such that the flowout $\Sigma_y$ of $S$ with interval $I_y$ is such that there exists $x \in S$ and $t \in \iota_x$ such that $\Sigma_y(x,t) = y$. Then, consider the open interval $I = \cup_{y \in M}I_y$ containing $0$ and the map $\Sigma : S \times I \to M$ defined by, for $x \in S$,
\begin{equation*}
\Sigma(x,t) = \Sigma_y(x,t) \text{ if there exists } y \in M \text{ such that } t \in I_y.
\end{equation*}
It is clear that this is well-defined and is a flowout of $S$ with interval $I$. Moreover, $\Sigma$ is onto $M$. Supposing that $M$ is not compact, by \cref{compactalternative}, $\Sigma$ must be injective. Thus, because $\Sigma$ is a bijection and local diffeomorphism, $\Sigma$ must be a diffeomorphism.
\end{proof}

This enables us to establish the partition of unity in the present situation.

\begin{lemma}
Let $\{U_S\}_{S \in \mathcal{F}}$ be an open cover of $M$ such that for each $S \in \mathcal{F}$, $U_S \supset S$. Then, there exists a partition of unity $\{\psi_S : M \to \R\}_{S \in \mathcal{F}}$ adapted to $\nu$ on $M$ and subordinated to $\{U_S\}_{S \in \mathcal{F}}$.
\end{lemma}

\begin{proof}
First, assume that $M$ is connected and compact. For each $S \in \mathcal{F}$, an open interval $I_S$ containing $0$ with a flowout $\Sigma_S' : S \times I_S' \to M$ which is a diffeomorphism onto its open image in $M$. Now, for each $S \in \mathcal{F}$, ${\Sigma'_S}^{-1}(U_S)$ is a neighbourhood of $S \times \{0\}$ in $S \times I_S'$. Thus, we may choose an open interval $I_S \subset I_S'$ containing $0$ such that $\Sigma_S'(S \times I_S) \subset U_S$ and hence a flowout $\Sigma_S : S \times I_S \to M$ which is a diffeomorphism onto its open image in $M$ contained in $U_S$. Now, for each $S \in \mathcal{F}$, choose a smooth bump function $\psi_S' : I_S \to \R$ with $0 \leq \psi_S' \leq 1$ which is compactly supported and $\psi_S'(0) = 1$. Then we obtain the function $\psi_S : S \times I_S \to \R$ given by $\psi_S(x,t) = \psi_S'(t)$ which is compactly supported in $S \times I_S$. Moreover, if $t_S : S \times I_S \to \R$ denotes projection into $\R \supset I_S$, we have the relation
\begin{equation*}
d \psi_S \wedge dt_S = 0.    
\end{equation*}
By use of $\Sigma_S$, this then gives a function $\rho_S' : M \to \R$ which $1$ on $S$ and is compactly supported in $U_S$ and additionally satisfies
\begin{equation*}
d\rho_S' \wedge \nu = 0.
\end{equation*}
Now, the open subsets $V_S = \{\rho_S' \neq 0\} \subset U_S$ form a covering of $M$ and so there exists $N \in \mathbb{N}$ and leaves $S_1,...,S_N \in \mathcal{F}$ such that $(V_{S_i})_{i = 1}^N$ is a covering of $M$. Then we obtain the smooth positive function
\begin{equation*}
P = \sum_{i = 1}^N \rho_{S_i}.
\end{equation*}
Lastly, we define a family of smooth functions $\{\rho_{S} : M \to \R\}_{S \in \mathcal{F}}$ as follows. If $S = S_i$ for some $i \in \{1,...,n\}$, then set $\rho_S = \rho_{S_i}'/P$ and otherwise, set $\rho_S = 0$. Then, clearly this is a partition of unity subordinated to the open cover $\mathcal{O}$ which, by virtue of the identity $dP \wedge \nu = 0$, additionally satisfies
\begin{equation*}
d\rho_S \wedge \nu = 0
\end{equation*}
for all $S \in \mathcal{F}$.

Now assume that $M$ is connected and non-compact. In this case, we may fix an open interval $I$ containing $0$, a leaf $S^* \in \mathcal{F}$, and a flowout $\Sigma : S^* \times I \to M$ of $S^*$ which is a diffeomorphism. Then, we have the open cover $\{\Sigma^{-1}(U_S)\}_{S \in \mathcal{F}}$ of $S^* \times I$. Now, for $z \in I$, setting $S_z = \Sigma(S^*,z)$, $\Sigma^{-1}(U_{S_z})$ is a neighbourhood of $S^* \times \{z\}$. Thus, for each $z \in I$, (using compactness of $S^*$), we may choose an open interval $J_z$ of $z$ such that $S^* \times J_z \subset \Sigma^{-1}(U_{S_z})$. Now, fix a partition of unity $\{\psi_z' : I \to \R\}_{z \in I}$ subordinated to the open cover $\{J_z\}_{z \in I}$. For each $z \in I$, set $\psi_z : S^* \times I \to \R$ where $\psi_z(x,t) = \psi_z'(t)$. Then, for each $z \in I$, we have the identity
\begin{equation*}
d\psi_z \wedge dt = 0
\end{equation*}
where $t : S^* \times I \to \R$ is projection into $\R \supset I$. Using the diffeomorphism $\Sigma$, and in particular that $I \ni z \mapsto S_z \in \mathbb{F}$ is a bijection, this gives a partition of unity $\{\rho_S : M \to \R\}_{S \in \mathcal{F}}$ adapted to $\nu$ on $M$ and subordinated to $\{U_S\}_{S \in \mathcal{F}}$.

In summary, this shows that the proposition is true if $M$ is connected. The general case now follows by considering the connected components of $M$.
\end{proof}

We will now extend the foregoing to account for a general closed region containing both boundary and degenerate closed leaves.

\subsubsection{With boundary and axes}\label{step2}

Let $M$ be a manifold with boundary and $\nu$ be a closed 1-form such that $\partial M^*\nu = 0$. If $U \subset \mathcal{C}$ is an open subset partitioned by closed leaves and degenerate closed leaves of $\nu$, then the 1-form $U^*\nu$ is closed and satisfies $(\partial U)^*(U^*\nu) = 0$ since $\partial U = \partial M \cap U$. Moreover, the closed region $\mathcal{C}_U$ of $U^*\nu$ is all of $U$. Let $F_U$ denote the set of closed leaves and degenerate closed leaves in $U$. Let $P(U)$ denote the following statement: \begin{quote}
``For any open cover $\{U_{L}\}_{L \in F_U}$ of $U$ such that $U_{L} \supset L$ for all $L \in A$, there exists a partition of unity adapted to $\nu$ on $U$ and subordinated to $\{U_{L}\}_{L \in F_U}$."    
\end{quote}

\begin{lemma}\label{lem:addonK}
Let $C$ be a connected component of $\mathcal{C}$. Then $\interior C \backslash \mathcal{C}^1$ is a connected open submanifold. Let $K$ be a degenerate closed leaf contained in $C$. Then, for any neighbourhood $U_K$ of $K$, there exists a codimension 0 compact connected manifold $R$ with boundary such that $\partial R$ is a closed leaf of $\nu$ and $K \subset \interior R \subset U_K$. Moreover, $P(C \backslash K) \Rightarrow P(C)$.
\end{lemma}

The proof of \cref{lem:addonK} relies upon the following relatively well-known result.

\begin{proposition}\label{connectedcomplement}
Let $X$ be a connected manifold and $K$ be an embedded submanifold of $X$ of codimension at least $2$ which is also a closed as a subset of $X$. Then $X\backslash K$ is connected.
\end{proposition}

Although \cref{connectedcomplement} may hold without assuming that $K$ is a closed subset, it is sufficient for our purposes to ignore this greater generality. We now prove \cref{lem:addonK}.

\begin{proof}[Proof of \cref{lem:addonK}]
First observe that $\interior C = C \backslash \partial C$ is connected. Now, $\mathcal{C}^1 \cap \interior C$ is a union of (disjoint) degenerate closed leaves where each leaf has a neighbourhood isolating it from the rest. On the other hand, $\mathcal{C}^1 \cap C$ is closed in $C$ (being the critical set of $\nu$ in $C$) and thus $\mathcal{C}^1 \cap \interior C = \mathcal{C}^1 \cap C$ is closed in $\interior C$. Hence, any sub-union of these leaves is also a closed subset of $\interior C$. We can write a union  
\begin{equation*}
\mathcal{C}^1 \cap \interior C = K_0 \cup ... \cup K_{n-2}
\end{equation*} 
where $K_i$ is the union of all degenerate closed leaves of dimension $i$ for $i = 0,...,n-2$. In particular, each $K_i$ is an embedded submanifold of $C$ of dimension $i$ which is also a closed subset of $C$. Thus, we can use \cref{connectedcomplement} inductively (on $i$) to infer that $\interior C \backslash \mathcal{C}^1$ is a connected open submanifold. 

As preliminaries for the remaining, consider $C_0 = (\interior C \backslash \mathcal{C}^1) \cup K$ which is also a connected open submanifold from the above argument. Because $K$ is embedded in $C_0$, by the Tubular Neighbourhood Theorem, there exists a neighbourhood $U$ of $K$ diffeomorphic to the normal bundle $\mathcal{V}(K)$ of $K$ in $C_0$. Since $U$ deformation retracts to $K$ where $K^*\nu = 0$ is in particular exact, $\nu$ is exact in $U$. Moreover, $C_0\backslash K$ is connected by \cref{connectedcomplement}. Thus, \cref{boundarylessfoundation} gives the existence of a flowout $\Sigma$ of some closed leaf $S$ of $\nu$ in $C_0$ with some open interval $I$ containing $0$ which is a diffeomorphism $S \times I \to C_0 \backslash K$. Recall also that $\Sigma^*\nu = dt$ where $t : S \times I \to \R$ is projection into $\R \supset I$. Altogether, $\nu$ is also exact on $C_0\backslash K$. Lastly, because $U \cap (C_0\backslash K) = U \backslash S$ is connected, $\nu$ is exact on $C_0$. Again, because $K^*\nu = 0$, there exists a unique $f : C_0 \to \R$ satisfying $\nu = df$ and $f|_K = 0$. The formula $\Sigma^*\nu = dt$ shows that there exists at most one closed leaf $S'$ in $C_0$ for which $(f|_{C_0\backslash K})^{-1}(0) = S'$. If there is such an $S'$, set $C_1$ to be the connected component $C_1$ of $C_0 \backslash S'$ containing $K$. Otherwise, set $C_1 = C_0$.

We will now simultaneously show the existence of such $R$ in the lemma and prepare for the implication involving partitions of unity. For this, fix an open neighbourhood $U_K$ of $K$. 

The set $K$ is the set of global minima for $\tilde{f} = f|_{C_1}$ (or maxima and in which case, perform the trivial adjustments to the remainder of the proof). We may apply the Tubular Neighbourhood Theorem in $U_K \cap C_1$ to obtain a neighbourhood $\tilde{U}$ of $K$ contained in $U_K$ diffeomorphic to the normal bundle $\mathcal{V}_1(K)$ of $K$ in $C_1$. We may then fix a disk bundle $D(\mathcal{V}_1(K))$ of $\mathcal{V}_1(K)$ \cite{Mukherjee2015}; this is a fibre bundle over $K$ with typical fibre the (closed) unit disk $D$ and is thus compact. Moreover, $D(\mathcal{V}_1(K))$ is a smooth manifold with boundary. Using the diffeomorphism $\mathcal{V}_1(K) \to \tilde{U}$, we thus obtain a codimension $0$ compact manifold $\tilde{R}$ with boundary such that $K \subset \interior \tilde{R}$. Observe by the choice of $C_1$ that $\tilde{f}$ does not obtain its minimum of $0$ on $\partial \tilde{R}$. Denoting by $y$ the positive minimum of $\tilde{f}$ on $\partial \tilde{R}$, the value $y/2$ (for instance) is a regular value of $\tilde{f}$ in $\tilde{R}$. Thus, we obtain the compact manifold $R = \{x \in \tilde{R} : f(x) \leq y/2\}$ with boundary $\partial R = \{x \in \tilde{R} : f(x) = y/2\}$ is a compact regular level set of $\tilde{f}$ in $\tilde{R}$. It is then easy to see that $\partial R$ is a closed leaf of $\nu$. Now, consider a smooth function $h : \R \to \R$ with $0 \leq h \leq 1$ supported in $(-y/2,y/2)$ satisfying $h(x) = 1$ for $x \in (-y/3,y/3)$ and set $\psi_k = h \circ f : C \to \R$. This is a smooth function with $0 \leq \psi_k \leq 1$ which is compactly supported in $U_K$ with $\psi_K = 1$ in a neighbourhood of $K$.

Now, suppose the statement $P(C \backslash K)$ is true. Let $A$ denote the set of closed leaves and degenerate closed leaves in $C$ and $\{U_{\alpha}\}_{\alpha \in A}$ be an open cover of $C$ whereby $U_{\alpha} \supset \alpha$ for each $\alpha \in A$. Set $\tilde{A} = A \backslash \{K\}$ and consider the open cover $\{U_{\alpha} \cap (C\backslash K) \}_{\alpha \in \tilde{A}}$ of $C \backslash K$. Because $P(C \backslash K)$ holds, there exists a partition of unity $\{\tilde{\psi}_{\alpha}\}_{\alpha \in A}$ adapted to $\nu$ on $C\backslash K$ and subordinated to $\{U_{\alpha}\}_{\alpha \in A}$. Then, considering the neighbourhood $U_K$ of $K$ in the open cover $\{U_{\alpha}\}_{\alpha \in A}$, by the above, there exists a smooth function with $0 \leq \psi_k \leq 1$ which is compactly supported in $U_K$ with $\psi_K = 1$ in a neighbourhood of $K$. Thus, we get a partition of unity adapted to $\nu$ on $C$ subordinated to $\{U_{\alpha}\}_{\alpha \in A}$ given by $\{\psi_K\}\cup \{\psi_\alpha\}_{\alpha \in \tilde{A}}$ where
\begin{equation*}
\psi_\alpha = (1-\psi_K)\tilde{\psi}_{\alpha}, ~ \alpha \in \tilde{A}
\end{equation*}
and multiplication at points of $K$ is interpreted as $0$.
\end{proof}

\begin{lemma}\label{addonbdry}
Let $C$ be a connected component of $\mathcal{C}$ let $S$ be a connected component of $\partial C$ (and thus a closed leaf of $\nu$). Then, we have the implication $P(C \backslash S) \Rightarrow P(C)$.
\end{lemma}

\begin{proof}
This follows directly from considering a flowout from the boundary using a similar algebraic construction to that of the end of the proof of \cref{lem:addonK}.
\end{proof}

\subsubsection{Classification of the closed region and the toroidal region}

Our classification for $\mathcal{C}$ is the following.

\begin{lemma}\label{classificationofC}
Let $C$ be a connected component of $\mathcal{C}$. Then the following holds.
\begin{enumerate}
    \item If $C$ is without boundary and contains no degenerate closed leaves, then $C$ is either compact or there exists a regular closed leaf $S$ such that $C$ is diffeomorphic to $S\times I$ via $\Sigma$, a flowout of $S$ with interval $I$.
    \item Otherwise, there exists some regular closed leaf $S$ such that $C \backslash (\partial C \cup \mathcal{C}^1)$ is diffeomorphic to $S\times I$ via $\Sigma$, a flowout of $S$ with interval $I$. Furthermore, the combined sum of closed leaves in $\partial C$ and degenerate closed leaves contained in $C$ is at most $2$. If this sum is realised, then $C$ is compact.
\end{enumerate}
\end{lemma}

\begin{proof}[Proof of \cref{classificationofC,classification}]
The first case follows immediately from Lemma \ref{boundarylessfoundation}. For the second case, first observe that $\partial C \cup (\mathcal{C}^1 \cap C)$ is a closed subset of $C$. Hence, $C \backslash (\partial C \cup \mathcal{C}^1)$ is not compact because otherwise $\partial C \cup \mathcal{C}^1$ would be open, in turn implying that $\interior C \backslash \mathcal{C}^1$ has an proper subset that is both open and closed. This contradicts (at least) the connectedness of $\interior C \backslash \mathcal{C}^1$ from \cref{lem:addonK}. Hence, $\interior C \backslash (\partial C \cup \mathcal{C}^1)$ is non-compact and connected. By \cref{boundarylessfoundation}, this implies that $\interior C \backslash (\partial C \cup \mathcal{C}^1)$ is diffeomorphic to $S\times I$ via $\Sigma$, a flowout of some closed leaf $S \subset \interior C \backslash (\partial C \cup \mathcal{C}^1)$ with some open interval $I$ containing $0$.

We now claim the following: if there are at least 2 closed leaves in $\partial C$ and degenerate closed leaves in $C$, then there are exactly $2$.

\begin{proof}[Proof of the claim]
We will treat the case of there being at least two closed leaves in $\partial C$. The cases involving axes are similar: the role of a leaf in our case is played by $\partial R$ by suitably small choice of $R$ as in Proposition \ref{lem:addonK} in the other cases. With this, let $L_1,L_2$ be two distinct such leaves. By the first paragraph and \cref{flowoutexistence,flowoutproperties}, there exist flowouts connecting each of these closed leaves to a leaf in $\interior C \backslash (\partial C \cup \mathcal{C}^1)$ which generates all of $\interior C \backslash (\partial C \cup \mathcal{C}^1)$. From these connecting flowouts, it is easy to see that $L_1 \cup L_2 \cup \interior C \backslash (\partial C \cup \mathcal{C}^1)$ is compact, open, and connected in $C$. Hence,
\begin{equation*}
C = L_1 \cup L_2 \cup \interior C \backslash (\partial C \cup \mathcal{C}^1).
\end{equation*}
So $\partial C$ consists of exactly the two closed leaves $L_1$ and $L_2$ and there are no degenerate closed leaves.
\end{proof}

Proposition \ref{classification} now follows from the above proof after direct adaptations of \cref{flowoutexistence,flowoutproperties} to flowouts of leaves from $\partial \mathcal{T}$ and the diffeomorphism flowouts in the connected components of $\interior \mathcal{T}^2$ together with the following observations. Recall, in this situation $M$ is orientable and of dimension $3$ and of course, the regular tori (or closed leaves) in $\mathcal{T}$ are each 2-tori. Lastly, concerning axes in $\mathcal{T}$, the $R$ which exist in \cref{lem:addonK} must be solid tori. Indeed, the vector bundles of fixed dimension over the circle have only two isomorphism classes with only the trivial class being orientable. Moreover, from the proof of \cref{lem:addonK}, $\tilde{R}$ is a disk subbundle of such a vector bundle making $\interior \tilde{R}$ diffeomorphic to a orientable vector bundle over the circle. Thus, because $R$ is embedded in $\interior R$ which must be an open solid torus, $R$ may be embedded in $\R^3$ and has a torus boundary, making $R$ a solid torus. This establishes the claim.
\end{proof}


\section{An obstruction to the global existence of adapted 1-forms}\label{sec:obstructiontoglobaleta}

Reeb cylinders, and more generally Kuperberg plugs, were identified in \cite{Peralta-Salas2021} as obstructions to the existence of MHS fields in arbitrary geometries. In this section, we will use a similar argument to show that the existence of Reeb cylinders are an obstruction to the global existence of an adapted 1-form. 

\begin{definition}
A \emph{Reeb cylinder} is a pair $(C,X)$ where $C$ is an oriented cylinder and $X$ is a non-vanishing vector field on $C$ tangent to $\partial C$ such that, considering the connected components $C_0,C_1$ of $\partial C$ whereby, as a 1-chain, $\partial C = C_0 - C_1$, the following hold:
\begin{enumerate}
    \item $X$ is positively (respectively negatively) oriented along the 1-chain $C_0$ and $C_0$ is the alpha-limit of each orbit except $C_1$.
    \item $X$ is negatively (positively) oriented along the 1-chain $C_1$ and is the omega-limit for every orbit except $C_0$.
\end{enumerate}
\end{definition}

The following elementary principle shows why Reeb cylinders are obstructions.

\begin{proposition}\label{prop:Reebcylinder}
Let $M$ be a 3-manifold with (possibly empty) boundary and $(B,\nu,\mu)$ a flux system on $M$. Suppose that $(B,\nu,\mu)$ possesses a non-trivial Reeb cylinder: that there exists a cylinder $C$ immersed in $M$ which $B$ is tangent to and such that $C$ together with the induced vector field $X$ on $C$ forms a Reeb cylinder, and the critical points of $\nu|_C$ are nowhere dense on $C$. Then, there cannot exist a 1-form adapted to $(B,\nu,\mu)$ on $M$.
\end{proposition}

\begin{proof}
Suppose that $(B,\nu,\mu)$ possesses a Reeb cylinder $(C,X)$. Then, observe that $C^*\nu$ is closed on $C$. Moreover, the periodic orbits of $B$ assumed on $\partial C$ generate the homology of $C$. Hence, from the fact that $\nu(B) = 0$, we get that $C^*\nu = df$ is exact on $C$. However, because of the orbit structure of $X$ and $df(X) = 0$, we have that $f$ is constant. Thus, $C^*\nu = 0$. 

Now, suppose that $\eta$ is a 1-form on $M$ such that $d\eta \wedge \nu = 0$. We will show that we cannot have $\eta(B)|_{\partial C} > 0$ and thus there are no adapted 1-forms. Indeed, at points $x \in C$ where $\nu|_x \neq 0$, because $C^*\nu|_x = 0$, we obtain that $C^*d\eta|_x = 0$. Since such points are dense in $C$ by assumption, we conclude that $C^*d\eta = 0$ so that $C^*\eta$ is a closed 1-form on $C$. Now, regarding $C$ as a 2-chain, Stokes' Theorem says
\[0 = \int_{C}d \eta = \int_{C_0} \eta - \int_{C_1} \eta.\]
However, if $\eta(B)|_{\partial C} > 0$, then by the orientation of $X$ along $\partial C$, accounting for the two cases of orientation of the periodic orbits, the quantity
\[\int_{C_0} \eta - \int_{C_1} \eta\]
is either strictly positive or strictly negative. Thus, $\eta(B)|_{\partial C} > 0$ cannot hold.
\end{proof}

We will now exhibit an example of a flux system on the solid torus which possesses two (non-trivial) Reeb cylinders as depicted in \cref{fig:thereebtorus}. Despite this implying that an adapted 1-form $\eta$ cannot globally exist, this example still has a large amount of additional structure such as a global symmetry. The example owes its additional structure to the method of its construction, which was by explicitly carrying out steps suggested in \cite[Lemma 5]{Kuperberg1996}.

\begin{figure}[h!]
    \centering
    \includegraphics[scale=1]{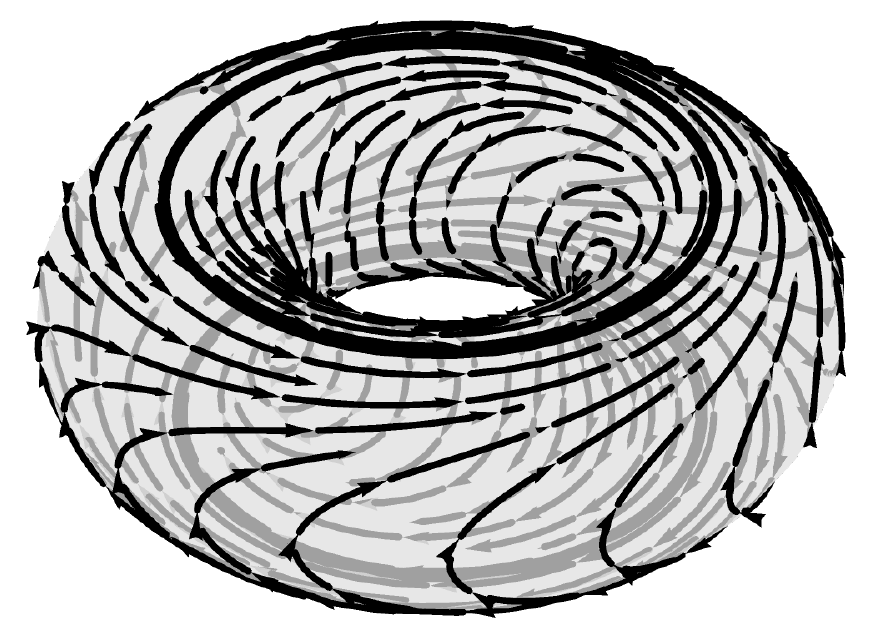}
    \caption{The invariant torus $x^2+y^2=1$ of the flux system $(B,\mu,dH)$ in \cref{theexample} embedded into $\mathbb{R}^3$. The invariant torus $x^2+y^2 = 1$ contains two Reeb cylinders of the flux system which intersect in the periodic orbits at the top and the bottom of the torus.}
    \label{fig:thereebtorus}
\end{figure}

\begin{example}\label{theexample}
Set $M = D \times \mathbb{R}/\mathbb{Z}$ where $D = \{(x,y) : x^2+y^2 \leq \sqrt{2} \}$ is the closed disk of radius $\sqrt{2}$ and use the standard coordinates $(x,y,\varphi)$ with $\varphi : M \to \mathbb{R}/\mathbb{Z}$ denoting projection onto the final factor of $M$. Consider the vector field $B$ given by
\[B = -H_y \partial_x + H_x \partial_y + f \partial_\varphi\]
where the Hamiltonian component of $B$ is generated by the function
\[H = (x^2+y^2-1)(1+ y^2(x^2+y^2-2))\]
and the toroidal component of $B$ is given by the function 
\[f = y + (x^2+y^2-1).\]

Observe that $B$ is divergence-free with respect to the volume form $\mu = dx \wedge dy \wedge d\varphi$. Furthermore, the vector field $u = \partial_{\varphi}$ is also divergence-free with respect to $\mu$ and satisfies $\iota_u\iota_B = dH$. It follows that $u$ is a symmetry of the flux system $(B,\mu,dH)$ and, since $H$ is constant and regular on $\partial M$, the flux system is tangential. 

We now discuss the topological structure of $H$ and $B$. Concerning $H$, it suffices to consider the function $h$ defined by the same formula as $H$ except on $D$ and we have the following.

\begin{proposition}
The critical points of $h$ are $(0,0)$ and $(0,\pm 1)$. Moreover, the critical point $(0,0)$ constitutes the level set $h^{-1}(-1)$ as the minimising set for $h$. The critical points $(0,\pm 1)$ belong to the level set $h^{-1}(0) = \{x^2+y^2 = 1\}$ and the remaining level sets are regular for $h$ and are circles. 

Thus, the level sets are nested in the sense that: for each level set $h^{-1}(z)$ with $z \neq -1$, $h^{-1}(z) = \partial A_z$ where $A_z = \{H \leq z\}$ is (diffeomorphically) a disk and for $z' < z$, we have the containment $A_{z'} \subset \interior A_z$ and in particular $h^{-1}(z') \subset \interior A_z$ so that $h^{-1}(z)$ ``contains $h^{-1}(z')$ in its interior".
\end{proposition}

\begin{proof}
First, note that
\[h_x = 2 x (2 y^4 + (2 x^2-3) y^2 + 1) ,\qquad h_y = 2 y (x^4 + (4 y^2 - 3)x^2 + 3 (y^2-1)^2).\]
In particular, the factor
\[x^4 + (4 y^2 - 3)x^2 + 3 (y^2-1)^2 = 3(y^2-1)^2+4x^2(y^2-1)+x^4+x^2\]
is a quadratic in $y^2-1$ with discriminant $4x^2(x^2-3)$. For $(x,y) \in D$, this discriminant is non-negative if and only if $x = 0$. From here, one may easily deduce that the only critical points of $h$ are $(0,0)$ and $(0,\pm 1)$ whereby $h(0,0) = -1$ and $h(0,\pm 1) = 0$. Clearly $h|_{\partial D} = 1$ and so $(0,0)$ is the unique global minimiser of $h$ so that $h^{-1}(-1) = \{(0,0)\}$. Moreover, $h$ may alternatively be written as
\[h = (x^2+y^2-1)(y^4 + (x^2-2)y^2 + 1)\]
and a similar quadratic argument as before shows that $h^{-1}(0) = \{x^2+y^2 = 1\}$. Now, the level set $h^{-1}(z)$ for $z \neq 0,1$ is regular. Let $L$ be a connected component of $h^{-1}(z)$ so that $L$ is (diffeomorphically) a circle and let $A$ be the unique disk with $\partial A = L$. Then, $\interior A$ must contain a critical point of $h$. If $\interior A$ contains one the $(0,\pm 1)$ critical points, then because $L \cap \{x^2+y^2=1\} = \emptyset$, we have that $A \cap \{x^2+y^2=1\} = \interior A \cap \{x^2+y^2=1\}$ and thus $\interior A \cap \{x^2+y^2=1\}$ is closed and open in $\{x^2+y^2=1\}$ and thus $\interior A$ contains all of $\{x^2+y^2=1\}$ and thus all of $\{x^2+y^2 < 1\}$ and in particular the critical point $(0,0)$. Thus, in general $\interior A$ contains $(0,0)$. Hence, by compactness, of $\{(t,0) : t\geq 0\} \cap A$, have a non-empty intersection $\partial A \cap \{x \geq 0, y = 0\} \neq \emptyset$. Hence, every connected component of every level set intersects the curve $\gamma(t) = (t,0)$ ($0 \leq t \leq \sqrt{2}$). On the other hand, $h(\gamma(t)) = t^2-1$ which is injective on $\{0 \leq t \leq \sqrt{2}\}$. Hence, the level sets of $h$ are connected.
\end{proof}

Now addressing the topology of $B$, first observe that $B$ has no zeros: for the critical set by the above of $H$ is $\{(0,\pm 1,\varphi),(0,0,\varphi): \varphi \in \mathbb{R}/\mathbb{Z}\}$ and there we have the equalities
\[f(0,0,\varphi) = -1, \qquad f(0,\pm 1, \varphi) = \pm 1, \qquad \varphi \in \mathbb{R}/\mathbb{Z}.\]
Lastly, we claim that the cylinders
\[ C_{\pm} = \{(x,y,\varphi) \in M\,|\, x^2+y^2 = 1,\, \pm x \geq 0 \}\]
are Reeb cylinders. Indeed, from the fact that $H$ is a first integral of $B$ and $C_{\pm} \subset H^{-1}(0)$ with $dH|_{\interior C_{\pm}}$ being non-vanishing, we have that $B$ is tangent to $C_{\pm}$. In fact, concerning $\partial C_{\pm}$, one may compute that
\[ B|_{(0,\pm 1,\varphi)} = \pm  \partial_{\varphi}, \qquad \varphi \in \mathbb{R}/\mathbb{Z} \]
so that $\partial C_{\pm}$ consists of two periodic orbits of the required opposite orientation for $C_{\pm}$ to be a Reeb cylinder. Lastly, on $\interior C_{\pm}$, using that $x^2 = 1-y^2$, we have
\begin{equation*}
dy(B) = H_x = 2x^3
\end{equation*}
and $\pm2x^3 > 0$ on $C_{\pm}$. Hence, we obtain the required alpha and omega limit property of the orbits in $\interior C$. In conclusion, $C$ is a Reeb cylinder.
\end{example}

To summarise, \cref{theexample} contains (non-trivial) Reeb cylinders and this obstructs the existence of an adapted 1-form. Despite this obstruction, \cref{theexample} has many nice properties, such as being non-vanishing, admitting a symmetry, and possessing nested invariant tori. The resolution between this example and our main result \cref{thm:etaexists} is that $H$ has critical points at the periodic orbits of the Reeb cylinder: if $H$ had no critical points, there would be a contradiction with \cref{etaonaxisnbhd}.


\section{Flux-coordinates and rectifiability}\label{sec:flux-coordinates}

In the study of magnetically confined plasma, it is often convenient to use special coordinates adapted to the magnetic field, so called \emph{magnetic} or \emph{flux-coordinates} \cite{dhaeseleer-1991,helanderTheoryPlasmaConfinement2014}. The main advantage of flux-coordinates is that field-lines are straight. From a computational point of view, magnetic differential equations in flux-coordinates have a simple form and can be solved by means of Fourier series in the corresponding poloidal and toroidal angles. Most notably, in the study of magneto-hydrostatics (MHS), particular choices of flux-coordinates are extensively used, such as Hamada coordinates \cite{hamadaHydromagneticEquilibriaTheir1962}, PEST coordinates~\cite{grimmComputationMagnetohydrodynamicSpectrum1976}, Boozer coordinates~\cite{boozer-1981}, etc. The methods presented in this paper provide the means of proving the existence of flux-coordinates without assuming that the magnetic field is MHS nor that the magnetic field is quasi-symmetric, nor that it has any symmetry by assumption.

In this section, quick and rigorous proofs of the existence of flux-coordinates are given for flux systems in general (\cref{cor:toruscoords,cor:axiscoords}). It is highlighted that flux-coordinates all come about because of conformal symmetries; different coordinates correspond to different pairs of commuting vector fields.
It is important to note that the full machinery of \cref{thm:etaexists} is not necessary. We only need the technical \cref{etaontubenbhd}, the elementary \cref{prop:etaonarbaxis}, and the algebraic \cref{thm:uexistsGivenEta}. The work in \cite{burby2021integrability} is a particular case of the results in this section. In \cref{sec:near a torus}, we will focus on flux-coordinates near regular tori. In section \cref{sec:axiscoords}, we consider flux-coordinates near axes relying heavily on the work done in \cite{burby2021integrability}. From this point on in this section, fix a 3-manifold $M$ with boundary.

To establish flux-coordinates, in relation to symmetries, we remark the following.

\begin{remark}\label{rmk:conformalsymmetryoffluxsystems}
If $(B,\nu,\mu)$ is a flux system on a manifold $M$ with boundary and $\eta$ is an adapted 1-form, then \cref{thm:uexistsGivenEta} implies that that $(\tB,\nu,\tmu)$ given by
\[\tB = B/f, \qquad \tmu = f\mu, \qquad f = \eta(B)\]
is also a flux system on $M$.
Moreover, the vector field $u$ satisfying
\[\iota_u \mu = (\nu \wedge \eta)/f\]
is such that
\[\iota_u \iota_{\tB} \tmu = \iota_u \iota_B \mu = \nu, \qquad [u,\tB] = 0.\]

An adapted 1-form is locally accessible for many flux systems $(B,\nu,\mu)$: if the flux system is tangential and $L$ is a regular torus, or if $L$ is a periodic orbit of $B$ contained in $\interior M$, by \cref{etaontubenbhd,prop:etaonarbaxis} there exists a neighborhood $\U$ of $L$ and a closed 1-form $\eta \in \Omega^1(\U)$ such that $\eta(B) > 0$. Hence, $\eta$ is adapted to $(B,\nu,\mu)$ in $\U$ and thus, one may rescale the tangential flux system on $\U$ and obtain a symmetry as concluded above. 
\end{remark}

In plasma physics, there are choices of an adapted 1-form which may canonically be available. For instance, if $M$ is an oriented Riemannian 3-manifold with boundary, then a non-vanishing divergence-free vector field $B$ satisfying
\[B \cdot \nabla \psi = 0, \qquad J \cdot \nabla \psi = 0,\]
where $\psi$ is a smooth function and $J = \nabla \times B$, has $\eta = B^{\flat}$ as an adapted 1-form. Then, $\tB=B/B^2$ and $u=\nabla\psi\times B/B^2$ are commuting vector fields. The resulting coordinates (when $B$ and $J$ are tangent to $\partial M$) on regular tori generated by $\psi$ are known as Boozer-coordinates~\cite{boozer-1981}. This was observed in \cite{perrella2022rectifiability} and in \cite{burby2021integrability} for the case that $B$ is MHS; namely $J \times B = \nabla \psi$. 

Another relevant and common adapted 1-form in the context of toroidal confinement is $\eta=d\varphi= (-ydx+xdy)/(x^2+y^2)$, where, in an abuse of notation, $\varphi(x,y,z)=\arctan(y/x)$ is the geometric toroidal angle of a solid torus $M$ embedded in $\R^3$ away from the $z$-axis. The 1-form $d\varphi$ is closed (but deceivingly not exact) and the assumption that $B^\varphi=d\varphi(B)>0$ is valid on physical grounds (the toroidal component is usually dominant in fusion devices such as tokamaks and stellarators). The related conformal symmetry of $B$ is
\[u = \frac{\nabla\psi\times \nabla \varphi}{B^\varphi}\]
and is purely poloidal: the integral curves lie on constant $\varphi$-lines in $S$ (and are hence closed). Consequently, the geometric toroidal angle $\varphi$ can serve as a coordinate, and a poloidal angle (another independent coordinate), can be found to make the field-lines of $B$ straight (or, following \cite{perrella2022rectifiability}, semi-rectified on $S$). The resulting coordinates are known as PEST coordinates~\cite{grimmComputationMagnetohydrodynamicSpectrum1976}. They are adopted by many MHS codes including VMEC~\cite{hirshmanSteepestDescentMoment1983} extensively used for stellarator optimisation.

We will now discuss how symmetries give rise to flux-coordinates and write down the consequences for general flux systems.

\subsection{Near a torus}\label{sec:near a torus}

In this section, we restrict to tangential flux systems. We remark that if one begins with a flux system on $M$ which is not tangential, then restriction to $\interior M = M \backslash \partial M$ produces a tangential flux system on $\interior M$.

The following theorem can be viewed as the general principle behind guaranteeing the existence of flux-coordinates near a torus. At the end of this section, we write down the consequences of this theorem for tangential flux systems.

\begin{theorem}\label{thm:SymmetryIsRectifiability}
    Let $(B,\nu,\mu)$ be a tangential flux system on $M$ and suppose that $S$ is a regular torus. Then, the following are equivalent. 
    \begin{enumerate}
        \item There exists a neighbourhood $\U$ of $S$ and a vector field $u$ on $\U$ satisfying $\iota_u \iota_B \mu = \nu$ and $[u,B]=0$.
        \item There exists a neighbourhood $\U$ of $S$, a vector field $u$ on $\U$ satisfying $\iota_u \iota_B \mu = \nu$, and a diffeomorphism $\Phi = (\theta_1,\theta_2,z) : \U \to \T^2\times I$ where $I$ is some interval such that
	\begin{equation}\label{rectified}
        \begin{split}
            \nu &= d\psi(z),\\
            B &= a(z) \partial_{\theta_1} + b(z) \partial_{\theta_2},\\
            u &= c(z) \partial_{\theta_1} + d(z) \partial_{\theta_2}. 
        \end{split}
        \end{equation}
        in $\U$, where $\psi,a,b,c,d : I \to \R$ are smooth.
    \end{enumerate}
\end{theorem}

\begin{proof}
Because $S$ is a regular torus, by \cref{flowoutexistence} and \cref{flowoutexistence}, there exists a neighborhood $U \subset \U$ of $S$ on which $\nu = d\psi$ for some $\psi \in C^{\infty}(U)$ whereby $S$ is a regular level set of $\psi$. We have that $[u,B] = 0$ and $u,B$ are tangent to the level sets $\psi$ (and $\partial M$). We conclude that 2 therefore holds by standard Arnol'd-Liouville argument (the general argument is given in \cite{arnold2013mathematical} and the current setting is given in \cite[Prop.~26]{perrella2022rectifiability}). The converse is a direct computation.
\end{proof}

Following \cite{perrella2022rectifiability}, if Equation \eqref{rectified} in \cref{thm:SymmetryIsRectifiability} is satisfied, we say that $\Phi$ \emph{rectifies $(B,\nu,\mu)$ and $u$}. 

It is often important in plasma physics to instead work with the 2-form $\beta = \iota_B \mu$. The following is a well-known reformulation of the aforementioned rectifiability condition.

\begin{proposition}\label{prop:RectifiabilityisHamada}
Let $(B,\nu,\mu)$ be a tangential flux system on $M$, $S$ be a regular torus, and $\U$ be a neighborhood of $S$ for which $\nu|_{\U}$ is non-vanishing. Let $u$ be a vector field on $\U$ satisfying $\iota_u \iota_B \mu = \nu$, and $\Phi = (\theta_1,\theta_2,z) : \U \to \T^2\times I$ be a diffeomorphism where $I$ is some interval. Then, the following are equivalent.
\begin{enumerate}
    \item The diffeomorphism $\Phi$ rectifies $(B,\nu,\mu)$ and $u$.
    \item The diffeomorphism $\Phi$ is such that the following is satisfied.
    \begin{equation}\label{eq:HamdaForm}
    \begin{split}
        \nu &= d\psi(z) \\
        \beta := \iota_B\mu& = dF(z) \wedge d\theta_1 - dG(z)\wedge d\theta_2,\\
        \vartheta := \iota_u\mu & = dK(z) \wedge d\theta_1 - dL(z) \wedge d\theta_2,
    \end{split}
    \end{equation}
    in $\U$, where $\psi,F,G,K,L : I \to \R$ are smooth.
    \end{enumerate}
\end{proposition}
Following \cite{burby2021integrability}, the coordinate functions $(\theta_1,\theta_2,z)$ satisfying Equation \ref{eq:HamdaForm} in \cref{prop:RectifiabilityisHamada} are \emph{Hamada coordinates for $(B,\nu,\mu)$ and $u$}.

\begin{proof}[Proof of \cref{prop:RectifiabilityisHamada}]
Assume that $\nu = d\psi(z)$ for some smooth $\psi : I \to \R$. Write
\begin{align*}
    B &= B^1\partial_{\theta_1} + B^2 \partial_{\theta_2},\\
    u &= u^1 \partial_{\theta_1} + u^2 \partial_{\theta_2},\\
    \mu &= \rho d\theta_1\wedge d\theta_2 \wedge dz ,
\end{align*}
where $B^1,B^2,u^1,u^2 : \U \to \R$ are smooth and $\rho : \U \to \R$ is smooth and non-vanishing. Then, we have that
\begin{align*}
    \beta &= \rho(B^2 dz \wedge d\theta_1 - B^1 dz \wedge d\theta_2) \\
    \vartheta &= \rho (u^2 dz \wedge d\theta_1 - u^1 dz \wedge d\theta_2) \\
    \iota_u\iota_B \mu &= \rho (B^1 u^2 - B^2u^1) dz.
\end{align*} 
Now, because $\iota_u\iota_B \mu = \nu = d\psi(z) = \psi'(z) dz$, we have that 
\[\rho (B^1 u^2 - B^2u^1) = \psi'(z)\]
and since $\nu|_{\U}$ is non-vanishing, both $\psi'(z)$ and $(B^1u^2 - B^2u^1)$ are non-vanishing. 

Hence, if $B^1,B^2,u^1,u^2$ are functions of $z$, then so is $\rho = \psi'(z)/(B^2u^1 - B^1 u^2)$ and thus $\rho B^1, \rho B^2, \rho u^1, \rho u^2$ are functions of $z$. Conversely, if $\rho B^1, \rho B^2, \rho u^1, \rho u^2$ are functions of $z$, then so is the function
\[\frac{(\rho B^1) (\rho u^2) - (\rho B^2)(\rho u^1)}{\psi'(z)} = \frac{\rho^2(B^1 u^2 - B^2u^1)}{\rho (B^1 u^2 - B^2u^1)} = \rho\]
and thus $B^1,B^2,u^1,u^2$ are functions of $z$. This establishes the claim.
\end{proof}

By \cref{rmk:conformalsymmetryoffluxsystems}, \cref{thm:SymmetryIsRectifiability}, and the proof of \cref{prop:RectifiabilityisHamada}, the following is immediate.

\begin{corollary}\label{cor:toruscoords}
Let $(B,\nu,\mu)$ be a tangential flux system. Let $S$ be a regular torus. Then, there exists a neighborhood $\U$ of $S$ and a diffeomorphism $\Phi = (\theta_1,\theta_2,z) : \U \to \T^2\times I$ where $I$ is some interval such that
\begin{align*}
    \nu &= d\psi(z),\\     
    \beta := \iota_B\mu &= dF(z) \wedge d\theta_1 - dG(z)\wedge d\theta_2,\\
    B/f &= G'(z)/\rho(z) \partial_{\theta_1} + F'(z)/\rho(z) \partial_{\theta_2},\\ 
    f \mu &= \rho(z) d\theta_1 \wedge d\theta_2 \wedge dz 
\end{align*}
in $\U$, where $f : \U \to \R$, $\psi,\rho, F,G : I \to \R$ are all smooth and $f,\rho > 0$.
\end{corollary}

\subsection{Near an axis}\label{sec:axiscoords}

In this section, flux-coordinates are established in the neighbourhood of an axis. The notion of axis readily extends to flux systems which are not tangential. Indeed, the restriction to $\interior M$ gives a tangential flux system for any flux system on $M$. We will refer in this section to this generalisation of axis. The results will hold provided the axis is non-degenerate.

\begin{definition}\label{def:non-degenerateaxis}
An axis $\gamma \subset \mathcal{T}^1$ of a flux system $(B,\nu,\mu)$ is non-degenerate if there exists a neighborhood $\U$ of $\gamma$ for which $\nu = d\psi$ in $\U$ for some $\psi : \U \to \R$ and $\gamma$ is a Morse-Bott critical manifold of $\psi$.
\end{definition}

\begin{remark}[on \cref{def:non-degenerateaxis}]\label{rmk:inasolidtorus}
The definition is simplified by the following observations about axes of a flux system $(B,\nu,\mu)$.
\begin{enumerate}
    \item If $\gamma$ is an axis, by \cref{lem:addonK} and the proof of \cref{classificationofC,classification}, in any neighborhood $U$ of $\gamma$, there exists a solid torus $R \subset U$ with $R \backslash \gamma$ contained in the pre-toroidal region $\mathcal{T}^2$ with boundary $\partial R$ a regular torus of $(B,\nu,\mu)$. In particular, $\interior R$ is a neighborhood of $\gamma$ for which $\nu$ is exact and $\gamma$ is the critical set of $\nu|_{\interior R}$.
    \item By the Morse-Bott Lemma, an axis $\gamma$ of $(B,\nu,\mu)$ is non-degenerate if and only if there exists a neighborhood $\U$ of $\gamma$ and a diffeomorphism $\Phi = (x,y,\varphi) : \U \to D^2 \times \T$ such that 
    \[\nu = ds\]
    where $s = \tfrac12 (x^2+y^2)$.
\end{enumerate}
\end{remark}

Analogously to the existence of flux-coordinates near a regular torus, we have the following theorem on the existence flux-coordinates near a non-degenerate axis for any flux system $(B,\nu,\mu)$. Having made \cref{rmk:inasolidtorus}, the forward direction follows immediately from the proofs of Theorem V.6 on near-axis Hamada coordinates and Theorem IV.3 in \cite{burby2021integrability} in the so-called elliptic case. The converse is a direct application of Equation \eqref{eq:ConformalNoether}.
\begin{theorem}\label{thm:NearAxisHamadaCoords}
Suppose that $(B,\nu,\mu)$ is a flux system and that that $\gamma$ is an axis. Then, the following are equivalent.
    \begin{enumerate}
        \item There exists a tubular neighbourhood $\U$ of $\gamma$ and a vector field $u$ on $\U$ such that $\iota_u \beta = \nu$ and $[u,B]=0$ in $\U$ and $\gamma$ is non-degenerate.
        \item There exists a tubular neighbourhood $\U$ of $\gamma$, a vector field $u$ on $\U$ such that $\iota_u \beta = \nu$, and a diffeomorphism $\Phi = (x,y,\varphi) :\U\rightarrow D^2\times \T$ such that 
        \begin{equation}\label{eq:HamdaFormForAxis}
        \begin{split}
            \nu &= d\psi(s),\\
            \beta &=  F(s) dx \wedge dy + dG(s) \wedge d\varphi \\
            \vartheta &= \iota_u\mu = K(s) dx \wedge dy + dL(s) \wedge d\varphi,
        \end{split}
        \end{equation}
        in $\U$, where $F,G,K,L,\psi : [0,1/2) \to \R$ are smooth and $s = \tfrac12 (x^2 + y^2)$.
    \end{enumerate}
\end{theorem}

Following \cite{burby2021integrability} the coordinate functions $(x,y,\varphi)$ satisfying Equation \ref{eq:HamdaFormForAxis} in \cref{thm:NearAxisHamadaCoords} are known as \emph{near-axis Hamada coordinates for $(B,\nu,\mu)$ and $u$}. We can also reformulate this condition to one explicitly on the vector fields $B$ and $u$ as follows.

\begin{proposition}
Let $(B,\nu,\mu)$ be a flux system on $M$, $\gamma$ be a non-degenerate axis, and $\U$ be a neighborhood of $\gamma$ for which $\nu|_{\U}$ has $\gamma$ as its critical set. Let $u$ be a vector field on $\U$ satisfying $\iota_u \iota_B \mu = \nu$, and $\Phi = (x,y,\varphi) : \U \to D^2 \times \T$ be a diffeomorphism where $D^2$ is the open unit disk. Then, the following are equivalent.
\begin{enumerate}
    \item The diffeomorphism $\Phi$ gives near axis Hamada coordinates for $(B,\nu,\mu)$ and $u$
    \item The diffeomorphism $\Phi$ is such that the following is satisfied.
    \begin{align*}
        \nu &= d\psi(s),\\
        B &= a(s) (y\partial_x-x\partial_y) + b(s) \partial_\varphi,\\
        u &= c(s) (y\partial_x-x\partial_y) + d(s) \partial_\varphi. 
    \end{align*}
    in $\U$, where $a,b,c,d,\psi : [0,1/2) \to \R$ are smooth and $s = \tfrac12 (x^2+y^2)$.
    \end{enumerate}
\end{proposition}

The proof is essentially the same as that of \cref{prop:RectifiabilityisHamada}. In particular, as in the case of a regular torus, by \cref{rmk:conformalsymmetryoffluxsystems}, \cref{thm:NearAxisHamadaCoords}, the following is immediate.

\begin{corollary}\label{cor:axiscoords}
    If $\gamma$ is a non-degenerate axis of a flux system $(B,\nu,\mu)$ then there exists a tubular neighborhood $\U$ of $\gamma$, a diffeomorphism $\Phi = (x,y,\varphi) : \U \to D^2\times \T$ where $D^2$ is the open unit disk such that
    \begin{align*}
        \beta &= F(s) dx \wedge dy + dG(s) \wedge d\varphi,\\
        B/f &=  G'(s)/\rho(s) (y \partial_x - x \partial_y) + F(s)/\rho(s) \partial_\varphi,\\
        f \mu &= \rho(s) dx \wedge dy \wedge d\varphi 
    \end{align*}
    in $\U$, where $s = \tfrac12 (x^2+y^2)$ and $a,b,\rho: [0,1/2) \to \R$ are smooth and $\rho > 0$.
\end{corollary}

\section{Discussion}\label{sec:discuss}

This work considered flux systems on a manifolds with (possibly empty) boundary. \cref{thm:uexistsGivenEta} identified an algebraic condition, namely the existence of an adapted 1-form, in order to conclude that the flux system had a conformal symmetry. This can be viewed as a partial converse to a conformal Noether theorem (\cref{thm:ForwardNoetherTheorem}). Assuming that the flux system was tangential, in other words, certain natural boundary conditions were obeyed, \cref{thm:etaexists} showed that the toroidal region is an apriori open set on which an adapted 1-form exists and in turn, a conformal symmetry. This provides a converse to a conformal Noether theorem on the topologically simplest region generated by the flux system. However, there remains some open questions and interesting further directions from the perspective of both plasma physics, and of a purely mathematical view.

In the context of plasma physics, this work demonstrates that structural features of magnetic fields, such as the existence of certain flux-coordinates, or symmetries, can be deduced from a flux function (first integral) alone, and the stronger assumption that the field satisfies MHS is not required. This has a direct implication for Grad's Conjecture \cite{Grad1967} which loosely states, under some non-degeneracy assumptions appealing to plasma confinement, that an MHS field with nested flux surfaces must possess a continuous geometric symmetry. This paper reveals that these structural features are purely topological in nature and that, because they are present for a large family of divergence-free vector fields, they alone do not give any insight into Grad's conjecture.

It is not known whether symmetry and first integrals are essential for good confinement; chaotic fields are in some scenarios desirable for the exhaust of heat and deleterious impurities~\cite{pedersenExperimentalConfirmationEfficient2022}. In the face of Grad's conjecture, if symmetry and first integrals are mathematically impossible beyond Euclidean symmetry, the flavour of the question changes to how "good" or "close" fields are to "useful" or "approximate" symmetry. These questions are primordial for the viability of stellarators and other non-axisymmetric magnetic confinement devices.

Turning to a more mathematical view of this paper, a first theoretical question is that of the global existence of an adapted 1-form $\eta$, for any flux system $(B,\nu,\mu)$. In this work, \cref{thm:etaexists} was proved, guaranteeing for any flux-system the existence of an adapted 1-form in the toroidal region. This region is the most relevant for flux systems appearing in plasma physics. However, mathematically it is crucial to ask whether the toroidal region is the largest possible region on which $\eta$ exists. The answer is no. Indeed, the flux system $(B,\nu,\mu)$ on $M = \R^2 \times \R/\mathbb{Z}$ given in \cref{NonAutHamExample} with
\[ B = y\partial_x +(2x-4x^3)\partial_y + \partial_t, \qquad \nu = d\psi, \qquad \mu = dx \wedge dy \wedge dt\]
and $\psi = \tfrac12 y^2 - x^2(1-x^2)$, is such a counter-example. The toroidal region is $\mathcal{T} = M \backslash \{\psi = 0\}$ and yet $\eta = dt$ is a 1-form adapted $(B,\nu,\mu)$ which is defined globally. More generally, any flux-system satisfying the hypothesis of \cref{prop:etaExists} for which $M$ is not the toroidal region serves as a counter-example. Besides including \cref{NonAutHamExample}, this also includes contact systems and cosymplectic systems.

Having established the fact that $\eta$ can exist on a larger region than the toroidal region, the natural question is then what is the largest region $\eta$ can be guaranteed. Unfortunately, finding an adapted 1-form for a flux system in high generality appears difficult. One approach to start finding an adapted 1-form of a flux system $(B,\nu,\mu)$ in a general setting can be roughly summarised as follows. On the non-critical set $U = \{\nu \neq 0\}$, assuming that $(B,\nu,\mu)$ is tangential, the leaves $L$ of the foliation generated by $\nu$ in $U$ have the property that $B$ is divergence-free with respect to the area-form $\mu_L = L^*i_N \mu$ in $L$ where $N : U \to TM$ is a vector field satisfying $\nu(N) = 1$. To obtain an adapted 1-form in $U$, one needs a closed 1-form $\eta_L$ on $L$ such that $\eta_L(B) > 0$. This problem is solved in this paper when $L$ is a 2-torus by Kolmogorov's Theorem \cite{perrella2022rectifiability}. In general, finding $\eta_L$ is purely about structural properties of non-vanishing, area preserving vector fields on $L$ and thus may depend on the differential topology of $L$. Due to the non-vanishing and area preserving properties, the possible dynamics are reduced and hence the desired $\eta_L$ for each $L$ might be obtained. However, even if each $\eta_L$ is obtained, one still needs to glue each of the $\eta_L$'s in $U$ together. This involves the structure of the foliation and hence the structure of $\nu$. Restricting to the toroidal region made the gluing relatively easy because the $L$ were compact and flowouts preserving the foliation could be used. Such tools are perhaps not usable for general $L$.

If $\eta$ can be found on the region where $\nu$ is non-critical, it still remains to show $\eta$ exists on regions including the critical points. The most natural hypothesis for the critical region $K = \{\nu = 0\}$ (or a subsets thereof) is that it consists of integral curves of $B$ which are open in the relative topology of $K$. If these integral curves are 1-dimensional and periodic and contained in $\interior M$, then as observed in \cref{sec:flux-coordinates} there exists an adapted 1-form. This local construction has nothing to do with $\nu$ or the toroidal region and can be done anywhere there is such an integral curve. An interesting consequence of these local adapted 1-forms, both a mathematically and physically, is that the work in \cite{burby2021integrability} should directly imply that flux-coordinates exist for these more general periodic orbits under some Morse-Bott non-degeneracy assumptions. This remains to be carefully checked in the so-called hyperbolic case of their work.

However, this problem of gluing in general cannot be resolved. This was exhibited by the flux system in \cref{theexample} in \cref{sec:obstructiontoglobaleta}. An interesting feature of this example is it in fact possesses a global symmetry despite there not existing any global adapted 1-form. In particular, \cref{thm:uexistsGivenEta} is not the mechanism through which every symmetry of a flux system is generated. This shows that the problem of constructing a global adapted 1-form faces problems of orientation which are not faced by symmetries. Indeed, considering for instance a Reeb cylinder $(C,X)$ where
\[C = \mathbb{R}/2\pi \times [-\pi/2,\pi/2], \qquad X = \sin y \partial_x + \cos y \partial_y\]
with standard coordinates $x,y$, we have that $\partial_x$ is a symmetry of $X$ and its integral curves are positively oriented with those of $X$ along the boundary component $C_1 = \mathbb{R}/2\pi \times \{\pi/2\}$ but negatively oriented with those along the boundary component $C_0 = \mathbb{R}/2\pi \times \{-\pi/2\}$. This is impossible for a closed 1-form to achieve: more precisely, as seen in \cref{prop:Reebcylinder}, there cannot be a closed 1-form $\eta_C$ on $C$ such that $\eta_C(X) > 0$.

While adapted 1-forms are not the only mechanism producing symmetries, \cref{thm:uexistsGivenEta} generates all symmetries of a certain type. More precisely, if $(B,\mu,\nu)$ is a flux system with an adapted 1-form $\eta$, then \cref{thm:uexistsGivenEta} there exists a vector field $u$, a positive function $f$ and a 1-form $\eta'$ such that
\begin{equation}\label{eq:kindofsymmetry}
  i_ui_B \mu = \nu, \qquad \mathcal{L}_u f \mu = 0, \qquad \eta'(u) = 0, \qquad \eta'(B) > 0
\end{equation}
where of course, as discussed, the first two equations imply that $u$ is a conformal symmetry of $B$. Under our assumptions, $\eta'$ may be taken as $\eta$ and $f$ as $\eta(B)$. Conversely, if one starts with a symmetry $u$ for which there exists an $\eta'$ and an $f$ as in Equation \eqref{eq:kindofsymmetry}, then the 1-form $\eta = f\eta'/\eta'(B)$ is an adapted 1-form. Indeed, we have $[u,B/f] = 0$, $\eta(B) = f$, and $\eta(u) = 0$ so that
\[d\eta(u,B/f) = u(\eta(B)/f)-B(\eta(u))-\eta([u, B/f]) = 0-0-0 = 0\]
and consequently $d\eta \wedge \nu = 0$, as required.

One necessary condition for a conformal symmetry $u$ to posses such a 1-form $\eta'$ as in \cref{eq:kindofsymmetry} is that $u$ vanishes at the critical points of $\nu$. However, this is not sufficient, as the following example illustrates.

\begin{example}\label{symmetrywhichcanbemodded}
In $\R^3$, set $B = \partial_x$ and $u = (y^2+z^2)(\partial_x + z\partial_y-y\partial_z)$. Then $u$ and $B$ are divergence-free with respect to the standard volume-form $\mu = dx \wedge dy \wedge dz$ and additionally satisfy
\[\iota_{u} \iota_{B} \mu = d\psi\]
where $\psi = -\frac{1}{4}(y^2+z^2)^4$. The critical set of $\psi$ is $\{y = 0 = z\}$ which is nowhere dense in $\R^3$ and $u$ only vanishes on this critical set. However, there does not exist a 1-form $\eta'$ such that
\begin{equation*}
\eta'(u) = 0, \qquad \eta'(B) > 0.
\end{equation*}
Indeed, letting $\eta' = \eta^1 dx + \eta^2 dy + \eta^3 dz$ be a 1-form such that $\eta'(u) = 0$, then
\[0 = \eta'(u) = (y^2+z^2)(\eta'(B) + z\eta^2-y\eta^3)\]
and consequently, because $\{y = 0 = z\}$ is nowhere dense in $\R^3$,
\[\eta'(B) + z\eta^2-y\eta^3 = 0\]
and hence $\eta'(B)$ must vanish on the critical set of $\psi$. 

On the other hand, $(B,d\psi,\mu)$ possesses the adapted 1-form $dx$ and by \cref{thm:uexistsGivenEta}, we must have a symmetry field $\tilde{u}$ generated from $dx$. One finds that $\tilde{u} = (y^2+z^2)(z\partial_y-y\partial_z) = u - (y^2+z^2)B$. Thus, in this instance, a modification of the symmetry $u$ into another symmetry $\tilde{u}$ was necessary for there to exist an $\eta'$ such that $\eta'(\tilde{u}) = 0$ and $\eta'(B) > 0$.
\end{example}

Although in \cref{symmetrywhichcanbemodded} there was a way of modifying the given symmetry $u$ into a symmetry $\tilde{u}$ which posesses an $\eta'$ as in \cref{eq:kindofsymmetry}, this modification is not always possible: as \cref{theexample} has no adapted 1-form, yet still possess a symmetry.

In conclusion, there is still much to understand about the global existence of $\eta$ and more generally, on the global existence of a conformal symmetry $u$ for a given flux-system. Future work would do well to investigate for which domains $\eta$ or these conformal symmetries may exist. The global existence of a conformal symmetry, whether that is through $\eta$ or not, is required to ensure the conformal Noether theorem demonstrated in this paper is as strong as the well-known Noether theorem in symplectic geometry.

\section{Acknowledgements}

The Authors would like to thank the Max Planck Institute in Greifswald for their hospitality during which the core ideas of the paper were formulated. Nathan would like to thank the Simons Collaboration on Hidden Symmetries and Fusion Energy for their financial support during this stay in Greifswald. David Perrella would like to thank Daniel Peralta-Salas for pointing out the idea of \cref{theexample} and providing financial support for his visit to the Instituto de Ciencias Matem{\'a}ticas in Madrid. This paper was written while David Perrella received an Australian Government Research Training Program Scholarship at The University of Western Australia.

\section{Author Declarations}

\subsection{Conflict of Interest}

The authors have no conflicts to disclose.

\subsection{Author Contributions}

\textbf{David Perrella}: Conceptualization (lead), Writing – original draft, Visualization. \textbf{Nathan Duignan}: Conceptualization, Writing – original draft, Visualization. \textbf{David Pfefferlé}: Supervision, Writing – original draft (supporting).

\appendix


\section{Tangential distributions and foliations on manifolds with boundary}\label{app:distandfol}

This appendix serves as a reference for a formal treatment of distributions and foliations which are tangent to the boundary of the underlying manifold. The purpose of this is that the authors have found no formal treatment in the literature. In Appendix \ref{app:Tangential Stefan-Sussman}, tangential distributions and integrability are discussed along with a proof of what we will call The Tangential Stefan-Sussman Theorem. The Stefan-Sussman Theorem reduces to the Frobenius Theorem in the context of the latter. In Appendix \ref{app:intmanifolds and foliations} integrable tangential distributions are related with foliations. The treatment will be rapid due to the close similarity to the case without boundary presented in the books of Lee \cite{Lee2012} and Rudolph and Schmidt \cite{Rudolph2013}. We will largely follow the latter.


\subsection{Tangential Stefan-Sussman Theorem}\label{app:Tangential Stefan-Sussman}

Let $M$ be a manifold with boundary. We first introduce tangential distributions.

\begin{definition}
Fix a subset $D \subset TM$. A $D$-section is a (smooth) vector field on $M$ taking values in $D$ and a local $D$-section is a local vector on $M$ taking values in $D$. The subset $D$ is said to be a distribution on $M$ if for all $x \in M$
\begin{enumerate}
    \item $D_x \coloneqq D \cap T_x M$ is a vector subspace of $T_x M$,
    \item and for all $v \in D_x$, there exists a $D$ section $X$ such that $X|_x = v$.
\end{enumerate}
In this case, for $x \in M$, the number $\dim D_x$ is called the rank of $D$ at $x$. Lastly, $D$ is said to be tangential if for all $x \in \partial M$, the vectors $v \in D_x$ are tangent to the boundary.
\end{definition}

We now introduce integrability. For clarity, in what follows, we will abuse notation by omitting pullbacks and pushfowards of tangent spaces via inclusions.

\begin{definition}
Fix a distribution $D$ on $M$. A connected submanifold $L \subset M$ is said to be an integral manifold of $D$ if $T_x L = D_x$ for all $x \in L$. The distribution $D$ is said to be integrable if through all points $x \in M$, there passes an integral manifold $L$ of $D$.
\end{definition}

\begin{remark}
An integrable distribution is necessarily tangential.
\end{remark}

We will now introduce the other notions for the Tangential Stefan-Sussman Theorem (Theorem \ref{Tangential Stefan-Sussman}).

\begin{definition}

Let $D$ be a tangential distribution $D$. Then $D$ is called
\begin{enumerate}
    \item involutive if whenever $X,Y$ are $D$-sections, so is $[X,Y]$.
    \item homogenous if for any local $D$-section $X$ and any point $(x,t)$ in the domain of the flow $\psi^X$ of $X$ (which is an open subset of $M \times \R$ because $X$ is tangent to the boundary), we have
    \begin{equation*}
    T\psi^X_t|_x D_x = D_{\psi^X_t(x)}.
    \end{equation*}
\end{enumerate}
\end{definition}

The remaining notion is that of an adapted chart, which is important for the connection with foliations. Set $n = \dim M$. We will denote by $\mathbb{H}^n = \{x^n \geq 0\}$ the upper half plane.

\begin{definition}
Let $D$ be a tangential distribution on $M$. Set $n = \dim M$. Let $x \in M$ and $r$ be the rank of $D$ at $x$. A local chart $(U,\varphi)$ on $M$ is said to be adapted to $D$ at $x$ if \begin{enumerate}
    \item $\varphi(x) = 0$ and there exists $\epsilon > 0$ such that $\varphi(U) = (-\epsilon,\epsilon)^n \cap \mathbb{H}^n$ if $x \in \partial M$ and $\varphi(U) = (-\epsilon,\epsilon)^n$ otherwise,
    \item the coordinate vector fields $\partial_{\varphi^1},...,\partial_{\varphi^r}$ are local $D$-sections
    \item for all $c \in \R^{n-r}$, the rank of $D$ along $U_c \coloneqq \varphi^{-1}((-\epsilon,\epsilon)^r \times \{c\})$ is constant.
\end{enumerate}
\end{definition}

With respect to these definitions, The Tangential Stefan-Sussman Theorem reads as follows.

\begin{theorem}[Tangential Stefan-Sussman]\label{Tangential Stefan-Sussman}
Let $D$ be a tangential distribution on $M$. Then the following are equivalent.
\begin{enumerate}
    \item[S1.] $D$ is integrable,
    \item[S2.] $D$ is involutive and has constant rank along integral curves of local $D$ sections,
    \item[S3.] $D$ is homogenous,
    \item[S4.] for every $x \in M$, there exists a local chart adapted to $D$ at $x$.
\end{enumerate}
\end{theorem}

We will now set out to prove Theorem \ref{Tangential Stefan-Sussman}; that the statements S1, S2, S3 and S4 are all equivalent. Using the construction of the double $D(M)$ (see \cite[Example 9.32]{Lee2012}), we may fix a manifold (without boundary) $\tilde{M}$ for which $M$ is a regular domain in $\tilde{M}$. 

\begin{remark}
Recall, $M \subset \tilde{M}$ being a regular domain means that the inclusion $i : M \subset \tilde{M}$ is a proper embedding and that $\dim \tilde{M} = n$. The embeddingness of the inclusion is equivalent to $M \subset \tilde{M}$ having the subspace topology. The properness of the inclusion is equivalent to $M$ being closed in $\tilde{M}$ \cite[Proposition 5.5]{Lee2012}. Because of this, and extension lemmas \cite[Lemma 2.26, Lemma 8.6]{Lee2012}, if $U$ is open in $\tilde{M}$ and $F$ is a function or local vector field defined on $U \cap M$, then there exists an extension $\tilde{F}$ of $F$ defined on $U$.
\end{remark}

Let $D$ be a tangential distribution on $M$. Define $\tilde{D} \subset T\tilde{M}$ by
\begin{align*}
\tilde{D}_x &= D_x \text{ for } x \in M, & \tilde{D}_x &= T_x \tilde{M} \text{ for } x \in \tilde{M} \backslash M.
\end{align*}
Using this as auxiliary, we prove Theorem \ref{Tangential Stefan-Sussman}.

\begin{proof}[Proof of Theorem \ref{Tangential Stefan-Sussman}]

Because Theorem \ref{Tangential Stefan-Sussman} is true for $\tilde{M}$, it suffices to establish that the subset $\tilde{D}$ is a distribution on $\tilde{M}$ and that for $i = 1,2,3,4$, $D$ satisfies statement Si on $M$ if and only if $\tilde{D}$ satisfies statement Si on $\tilde{M}$. 

To this end, let $X$ is a local $D$-section defined on $U\cap M$ where $U$ is some open subset of $\tilde{M}$. Because $M \subset \tilde{M}$ is a regular domain, there exists an extension of $X$ to $\tilde{X}$, a local vector field on $\tilde{M}$ defined on $U$. Moreover, any such extension is a local $\tilde{D}$-section. In particular, this shows that $\tilde{D}$ is a distribution on $\tilde{M}$. The equivalence between $D$ and $\tilde{D}$ satisfying S2 or S3 also follows from this observation and the naturality of Lie brackets and flows.

For S1, and another argument (Proposition \ref{regularity and maximality}) we have the following lemma.

\begin{lemma}\label{integral manifold correspondence}
The following are equivalent.
\begin{enumerate}
    \item $L$ is an integral manifold of $D$
    \item $L$ is an integral manifold of $\tilde{D}$ with $L \cap M \neq \emptyset$.
\end{enumerate}
\end{lemma}

\begin{proof}[Proof of Lemma \ref{integral manifold correspondence}]
For the less obvious direction, let $L$ is an integral manifold of $\tilde{D}$ with $L \cap (\tilde{M} \backslash M) \neq \emptyset$. Then $\dim L = n$ and so because the rank of $\tilde{D}$ is at most $n-1$ on $\partial M$, we have that $L \cap \partial M = \emptyset$. On the other hand $\interior M$ and $\tilde{M}\backslash M$ are open disjoint sets covering $\tilde{M} \backslash \partial M$. Thus, because $L$ is a connected subset of $M$, we must have that $L \subset \tilde{M} \backslash M$. Contrapositively, if $L$ is an integral manifold of $\tilde{D}$ with $L \cap M \neq \emptyset$, then $L \subset M$. Because $M$ is a regular domain in $\tilde{M}$, $L \subset M$ is a submanifold. Thus, $L$ is an integral manifold of $D$.
\end{proof}

Continuing with S1, observe that the connected components of $\tilde{M}\backslash M$ are integral manifolds of $\tilde{D}$ covering $\tilde{M} \backslash M$. It is then clear by Lemma \ref{integral manifold correspondence} that $D$ satisfies S1 on $M$ if and only if $\tilde{D}$ satisfies S1 on $\tilde{M}$.

For S4, first observe that for $x \in \tilde{M} \backslash M$, any chart $(U,\varphi)$ about $x$ with $U \subset \tilde{M}\backslash M$ satisfying $\varphi(x) = 0$ and $\varphi(U) = (-\epsilon,\epsilon)^n$ for some $\epsilon > 0$ is adapted to $\tilde{D}$ at $x$. Let $x \in \interior M$ then if $(U,\varphi)$ is a chart adapted to $D$ at $x$, $(U,\varphi)$ is a chart adapted to $\tilde{D}$ at $x$. Conversely, if $(U,\varphi)$ is a chart adapted to $\tilde{D}$ at $x$, we may choose $\epsilon' > 0$ sufficiently small so that $U' \coloneqq \varphi^{-1}((-\epsilon',\epsilon')^n) \subset \interior M$. Then, the restricted chart $(U',\varphi' : U \to (-\epsilon',\epsilon')^n)$ is adapted to $D$ at $x$.

Continuing with S4, suppose that $x \in \partial M$. Let $(U,\varphi)$ be a chart adapted to $D$ at $x$. We have $U = U' \cap M$ for some $U'$ open in $\tilde{M}$. The coordinate functions ${\varphi}^i : U \to \R$ may be extended to functions ${\varphi'}^i : U' \to \R$. Then, consider the map $\varphi' = ({\varphi'}^1,...,{\varphi'}^n) : U' \to \R^n$. Then by The Inverse Function Theorem, because $(U,\varphi)$ is a chart, there exists $\epsilon > 0$ so that $\varphi'$ restricts to a chart $(U'',\varphi'')$ on $U'' = \varphi'^{-1}((-\epsilon,\epsilon)^n)$. Because $(U,\varphi)$ is adapted to $D$ at $x$, and $\tilde{D}_y = T_y \tilde{M}$ for all $y \in \tilde{M}\backslash M$, $(U'',\varphi'')$ is a chart adapted to $\tilde{D}$ at $x$. 

Finally, let $(U,\varphi)$ be a chart adapted to $\tilde{D}$ at $x$. Here we adapt ideas in the proof of the Stefan-Sussman Theorem found in \cite{Rudolph2013}. Consider a chart $(V,\kappa)$ of $\tilde{M}$ at $x$ such that $V \cap M = \{\kappa^n \geq 0\}$, and $V \cap \partial M = \{\kappa^n = 0\}$. Setting $r$ to be the mutual rank of $D$ and $\tilde{D}$ at $x$, extend $\partial_{\varphi^1}|_x,...,\partial_{\varphi^r}|_x$ into a basis
\begin{equation*}
\partial_{\varphi^1}|_x,~...,~\partial_{\varphi^r}|_x,~v_{n-r},~...,~v_{n-1},~\partial_{\kappa^n}|_x
\end{equation*}
of $T_x \tilde{M}$ where the $v_{i}$ are tangent to $\partial M$. Then, we may extend the $v_{i}$ to vector fields $X_{i}$ on $\tilde{M}$. Then, consider the local vector fields $X_{i}'$ defined on $U \cap V$ where
\begin{align*}
X_i' &= \partial_{\varphi^i}, & i &= 1,...,r,\\ 
X_i' &= X_i - d\varphi^n(X_i)X_n', & i &= n-r,...,n-1, & X_n' &= \partial_{\kappa^n}.
\end{align*}
Considering the local flows $\psi^i \coloneqq \psi^{X_{i}'}$, there exists $\epsilon > 0$ sufficiently small so that the map $F : (-\epsilon,\epsilon)^n \to \tilde{M}$ given by
\begin{equation*}
F(t_1,...,t_n) = \psi^1_{t_1} \circ ... \circ \psi^n_{t_n}(x)
\end{equation*}
is well-defined. Because $X'_1|_x,...,X'_n|_x$ forms a basis, The Inverse Function Theorem ensures a sufficiently small $\epsilon' > 0$ so that $F|_{(-\epsilon',\epsilon')^{n}}$ is a diffeomorphism onto its open image. Moreover, since the $X_1,...,X_{n-1}$ is tangent to $\partial M$, an integral curve $\gamma_i$ of $X_i$ intersects $\partial M$ if and only if $\gamma_i \subset \partial M$. This invariance implies that
\begin{align*}
F((-\epsilon',\epsilon')^n \cap \mathbb{H}^n) &\subset \{\kappa^n \geq 0\}\\
F((-\epsilon',\epsilon')^n \backslash \mathbb{H}^n) &\subset \{\kappa^n < 0\}.
\end{align*}
thus, restriction of $F$ to $(-\epsilon',\epsilon')^{n} \cap \mathbb{H}^n$ gives a chart $(U',\varphi')$ of $M$ at $x$. Now, because the $X_1,...,X_r$ commute and the construction of $F$, it is easy to see that
\begin{align*}
\partial_{{\varphi'}^i} &= \partial_{\varphi^i}, & i &= 1,...,r.
\end{align*}
so that for each constant $c' \in \R^{n-r}$, there exists a constant $c \in \R^{n-r}$ such that
\begin{equation*}
\varphi'^{-1}((-\epsilon',\epsilon') \times \{c'\}) \subset \varphi'^{-1}((-\epsilon',\epsilon') \times \{c\}).
\end{equation*}
From these observations it immediately follows that $(U',\varphi')$ is an adapted chart to $D$ at $x$. This completes the equivalence between $D$ and $\tilde{D}$ satisfying S4 and thus the proof.
\end{proof}

We will now discuss integrable distributions in relation to foliations.
 
\subsection{Integral manifolds and foliations}\label{app:intmanifolds and foliations}

Let $M$ be a manifold with boundary. Just as in the case of manifolds (see \cite{Lee2012} and \cite{Rudolph2013}), the following definition serves as a specification of regularity of immersed submanifolds.

\begin{definition}
An (immersed) submanifold $L \subset M$ is said to be weakly embedded in $M$ if for any manifold $N$ and smooth map $F : N \to M$ such that $F(N) \subset L$, we have that the restricted map $f : N \to L$ is smooth.
\end{definition}

\begin{remark}
Just as in the case of manifolds without boundary \cite[Theorem 5.33]{Lee2012}, if $L \subset M$ is a weakly embedded submanifold, then $L$ has the unique topology and smooth structure making it into a submanifold of $M$. That is, if $L' \subset M$ is a submanifold having the same underlying set of points as $L$, then, as manifolds, $L' = L$.
\end{remark}

Let $D$ be an tangential distribution on a manifold $M$ with boundary. We will now introduce maximality of integral manifolds of $D$.

\begin{definition}
An integral manifold $L$ of $D$ is said to be maximal if for any integral manifold $K$ of $D$ with $L \cap K \neq \emptyset$, we have that $K \subset L$ is an open submanifold.
\end{definition}

To prove the following proposition which involves the definitions just introduced, fix a manifold $\tilde{M}$ for which $M\subset \tilde{M}$ is a regular domain. Moreover, we consider the associated distribution $\tilde{D}$ just as in Section \ref{app:Tangential Stefan-Sussman}.

\begin{proposition}\label{regularity and maximality}
The following holds.
\begin{enumerate}
    \item At every point $x \in M$, there passes a maximal integral manifold of $D$.
    \item The integral manifolds of $D$ are weakly embedded in $M$.
\end{enumerate}
\end{proposition}

\begin{proof}
First note that Proposition \ref{regularity and maximality} holds when $M$ has no boundary (see Proposition 3.5.15 and Theorem 3.5.17 in \cite{Rudolph2013}). In particular, it holds for the distribution $\tilde{D}$ on $\tilde{M}$. Part 1 is now easily established with Lemma \ref{integral manifold correspondence}.

For part 2, if $L$ is an integral manifold of $D$, then $L$ is an integral manifold of $\tilde{D}$ and is thus weakly embedded in $\tilde{M}$. So, if $N$ is a manifold and $F : N \to M$ is a smooth map with $F(N)\subset L$, then we have an induced smooth map $\tilde{F} : N \to \tilde{M}$ satisfying $\tilde{F}(N) \subset L$. Thus, the restricted map $\tilde{f} : N \to L$ is smooth. However, this is also the restricted map $f : N \to L$ from $F$. Thus, $L \subset M$ is weakly embedded. 
\end{proof}

We will now define the notion of foliation. As before, let $M$ be a manifold with boundary.

\begin{definition}\label{foliationdef}
A foliation is a partition $\mathcal{F}$ of $M$ into nonempty immersed submanifolds such that for every $x \in M$, there exists a chart $(U,\varphi)$ satisfying
\begin{enumerate}
    \item $\varphi(x) = 0$ and for some $\epsilon > 0$, $\varphi(U) = (-\epsilon,\epsilon)^n \cap \mathbb{H}^n$ if $x \in \partial M$ and $\varphi(U) = (-\epsilon,\epsilon)^n$ otherwise.
    \item the submanifolds $L \in \mathcal{F}$ are invariant under flows of the local vector fields $\partial_{\varphi^1},...,\partial_{\varphi^r}$ where $r = \dim L$.
\end{enumerate}
If $\mathcal{F}$ is a foliation, the elements of $\mathcal{F}$ are called leaves of the foliation.
\end{definition}

We then obtain a bijective correspondence between distributions and foliations as follows.

\begin{proposition}
Let $M$ be a manifold with boundary. The assignment of the family of maximal integral manifolds to an integrable distribution $D$ on $M$ defines a bijection between integrable distributions on $M$ and foliations on $M$.
\end{proposition}

\begin{proof}
Clearly 
\begin{equation*}
\mathcal{D} = \{D : D \text{ is an integrable distribution on } M\}
\end{equation*}
forms a set. Moreover, for each $D \in \mathcal{D}$, there exists a unique collection $\mathcal{F}_D$ of all maximal integral manifolds of $D$. Thus, we have a function $D \mapsto \mathcal{F}_D$. Concerning the image, if $D \in \mathcal{D}$, then Theorem \ref{Tangential Stefan-Sussman} shows that $\mathcal{F}_D$ is a foliation.

For injectivity, if $D,D' \in \mathcal{D}$ and $\mathcal{F}_D = \mathcal{F}_{D'}$, then for each $x \in M$, and $L \in \mathcal{F}_D$, we have
\begin{equation*}
D_x = T_x L = D'_x
\end{equation*}
concluding injectivity.

For surjectivity, let $\mathcal{F}$ be a foliation on $M$. Then, consider the subset $D \subset TM$ where for $x \in M$, setting $L \in \mathcal{F}$ to be the unique leaf containing $x$, $D_x = T_x L$. Definition \ref{foliationdef} shows that $D$ is an integrable distribution. We now show that $\mathcal{F}_D = \mathcal{F}$. Indeed, let $L \in \mathcal{F}_D$. Then by maximality, for each $L' \in \mathcal{F}$, if $L' \cap L \neq \emptyset$, then $L' \subset L$ is an open submanifold. Because distinct leaves are disjoint, and $L$ is connected, it follows that $L \in \mathcal{F}$. Thus, $\mathcal{F}_D \subset \mathcal{F}$. Lastly, because both $\mathcal{F}_D$ and $\mathcal{F}$ partition $M$ and each of their elements are non-empty, we conclude $\mathcal{F}_D = \mathcal{F}$ and therefore surjectivity.
\end{proof}

\bibliographystyle{apsrev4-2}
\bibliography{paper.bib}

\end{document}